\documentclass[reqno,11pt]{amsart}

\usepackage{amsmath,amssymb,amsthm,amsxtra} %
\usepackage{graphicx} 
\usepackage[foot]{amsaddr}

\makeatletter
\renewcommand{\email}[2][]{%
  \ifx\emails\@empty\relax\else{\g@addto@macro\emails{,\space}}\fi%
  \@ifnotempty{#1}{\g@addto@macro\emails{\textrm{(#1)}\space}}%
  \g@addto@macro\emails{#2}%
}
\makeatother

\usepackage[normalem]{ulem}
\usepackage{autonum}
\usepackage[utf8]{inputenc} 
\usepackage[T1]{fontenc}
\usepackage{enumitem}

\newcommand{\DET}[1]{}
\newcommand{\DS}{\displaystyle}

\newtheorem{theorem}{Theorem}[section] %
\newtheorem{lemma}[theorem]{Lemma}
\newtheorem{definition}[theorem]{Definition}

\newtheorem{proposition}[theorem]{Proposition}

\newtheorem{remark}[theorem]{Remark} %

  { \end{itemize} }
  { \end{enumerate} }

\newcommand{\BS}[0]{\backslash}

\newcommand{\FA}[1]{\ensuremath{\forall #1}}

\newcommand{\HW}[1]{} 

\renewcommand{\t}{\tilde}
\newcommand{\p}{\partial}
\renewcommand{\d}{\delta}

\newcommand{\SUS}{\subset}
\newcommand{\Rn}{\R^n}

\newcounter{pcounter}

\renewcommand{\AA}{{\mathcal A}}

\newcommand{\EE}{{\mathcal E}}

\newcommand{\HH}{{\mathcal H}}
\newcommand{\II}{{\mathcal I}}

\newcommand{\MM}{{\mathcal M}}
\newcommand{\OO}{{\mathcal O}}

\newcommand{\R}{\mathbb{R}}

\newcommand{\Z}{\mathbb{Z}}

\newcommand{\alp}{\alpha}

\newcommand{\gam}{\gamma}
\newcommand{\eps}{\epsilon}

\newcommand{\lam}{\lambda}
\renewcommand{\phi}{\varphi}
\newcommand{\ome}{\omega}
\newcommand{\sig}{\sigma}

\newcommand{\Gam}{\Gamma}
\newcommand{\Lam}{\Lambda}

\newcommand{\Ome}{\Omega}

\DeclareMathOperator*{\dist}{dist}

\DeclareMathOperator*{\spt}{spt}

\renewcommand{\iint}{\int\!\!\!\!\int}

\def\Xint#1{\mathchoice
{\XXint\displaystyle\textstyle{#1}}%
{\XXint\textstyle\scriptstyle{#1}}%
{\XXint\scriptstyle\scriptscriptstyle{#1}}%
{\XXint\scriptscriptstyle\scriptscriptstyle{#1}}%
\!\int}
\def\XXint#1#2#3{{\setbox0=\hbox{$#1{#2#3}{\int}$}
\vcenter{\hbox{$#2#3$}}\kern-.5\wd0}}

\def\dashint{\Xint-}

\newcommand{\lupref}[2]{\hspace{0ex} \stackrel{\eqref{#1}}{#2}} 

\usepackage{color}
\definecolor{verylightblue}{rgb}{0.95, 0.95, 0.95}  
\definecolor{lightblue}{rgb}{0.7, 0.7, 1}

\definecolor{eqyellow}{rgb}{0.9375,0.8984,0.5469}
\definecolor{subeqyellow}{rgb}{1,0.9373,0.8353}

\definecolor{darkcyan}{rgb}{0.0, 0.45, 0.95} 

\definecolor{mygreen}{rgb}{0.3, 0.6, 0.3} 
\definecolor{verylightgreen}{rgb}{0.95, 0.95, 0.95} 
\definecolor{verydarkgreen}{rgb}{0, 0.5, 0}
\definecolor{darkgreen}{rgb}{0.85, 0.85, 0.85}  
\definecolor{mydarkgreen}{rgb}{0, 0.5, 0} 

\definecolor{mybrown}{rgb}{0.85, 0.4, 0.3}
\definecolor{verylightbrown}{rgb}{0.98, 0.72, 0.58}
\definecolor{verydarkbrown}{rgb}{0.44, 0.26, 0.26}

\definecolor{orange}{rgb}{1, 0.5, 0}

\definecolor{BurntOrange}{rgb}{1,0.356,0}
\newcommand{\BurntOrange}{\color{BurntOrange}}

\definecolor{mydarkred}{rgb}{1,0.086,0.255}

\definecolor{RoseVYDP}{rgb}{0.84,0.086,0.255}

\definecolor{dgreen}{rgb}{0, 0.8, 0.5}     

\definecolor{CanaryBRT}{rgb}{1,0.76,0.26}

\definecolor{cyan}{rgb}{0, 1, 1}
\definecolor{verylightgray}{rgb}{0.95, 0.95, 0.95}
\definecolor{lightgray}{rgb}{0.8, 0.8, 0.8}
\definecolor{verylightred}{rgb}{1, 0.8, 0.78}
\definecolor{verylightyellow}{rgb}{0.99, 0.98, 0.5}

\newcommand{\blue}[1]{{\color{blue}{#1}}}

\newcommand{\ignore}[1]{{}}

\newcommand{\NN}[1]{\|#1\|}
\newcommand{\NNN}[2]{\|#1\|_{#2}}

\newcommand{\skp}[2]{(#1, #2)}

\newcommand{\NI}[1]{\|#1\|_{L^\infty}}

\newcommand{\NP}[2]{\|#1\|_{L^{#2}}}
\newcommand{\NPL}[3]{{\|#1\|_{L^{#2}(#3)}}}

\newcommand{\cciL}[1]{C_c^{\infty}(#1)}

\makeatletter
\newcommand{\VEC}[2][r]{
  \gdef\@VORNE{1}
  \left(\hskip-\arraycolsep%
    \begin{array}{#1}\vekSp@lten{#2}\end{array}%
  \hskip-\arraycolsep\right)}
\def\vekSp@lten#1{\xvekSp@lten#1;vekL@stLine;}
\def\vekL@stLine{vekL@stLine}
\def\xvekSp@lten#1;{\def\temp{#1}%
  \ifx\temp\vekL@stLine
  \else
    \ifnum\@VORNE=1\gdef\@VORNE{0}
    \else\@arraycr\fi%
    #1%
    \expandafter\xvekSp@lten
  \fi}
\makeatother

\renewcommand{\S}{{\mathbb S}}  %
\newcommand{\hk}[1]{{\blue{#1}}} %

\newcommand{\hkcc}[1]{}
\newcommand{\tbl}[1]{{\BurntOrange{$\langle\hspace{-0.3ex}\langle$#1$\rangle\hspace{-0.3ex}\rangle$}}} %

\newcommand{\T}{{\mathbb T}} %
\newcommand{\wtos}{\stackrel{*}{\rightharpoonup}} %
\newcommand{\Td}{{\T_\ell}}
\newcommand{\Rd}{\R^d}
\newcounter{margcount} 
\newcommand{\Gamto}{\stackrel{\Gam}{\to}}

\begin{document}

\title{Second order expansion for the nonlocal perimeter functional}

\author{Hans Knüpfer} 
\email[H. Knüpfer]{hans.knuepfer@math.uni-heidelberg.de}
\address[H. Knüpfer]{University of Heidelberg, MATCH and IWR, INF 205, 69120 Heidelberg, Germany}

\author {Wenhui Shi} 
\email[W. Shi]{wenhui.shi@monash.edu}
\address[W. Shi]{9 Rainforest Walk Level 4, Monash University, VIC 3800, Australia}

\begin{abstract}
  The seminal results of Bourgain, Brezis, Mironescu \cite{BBM01} and D\'avila
    \cite{Da02} show that the classical perimeter can be approximated by a
    family of nonlocal perimeter functionals. We consider a corresponding second
    order expansion for the nonlocal perimeter functional. In a special case,
    the considered family of energies is also relevant for a variational model
    for thin ferromagnetic films.  We derive the $\Gam$--limit of these
    functionals as $\eps \to 0$. We also show existence for minimizers with
    prescribed volume fraction.  For small volume fraction, the unique, up to
    translation, minimizer of the limit energy is given by the ball. The
    analysis is based on a systematic exploitation of the associated
    symmetrized autocorrelation function.
\end{abstract}


\maketitle

\tableofcontents


\section{Introduction and statement of main results}

The seminal results of Bourgain, Brezis, Mironescu \cite{BBM01} and D\'avila
\cite{Da02} show that the classical perimeter can be approximated by a family of
nonlocal perimeter functionals.  In a slightly reformulated setting this implies
that for any function $u\in BV(\Td,\{0,1\})$ we have
  \begin{align} \label{first-order} \dashint_{\Td} |\nabla u| \ dx \ = \
    \lim_{\eps\to 0} \frac 1{M_\eps}\int_{\Rd} K_\eps(z) \dashint_{\Td}
    \frac{|u(x+z)-u(x)|}{2|z|} \ dxdz,
  \end{align}
  for a family of kernels $K_\eps:\R^d\rightarrow [0,\infty)$,
  $K_\eps\in L^1(B_1)$, $K_\eps(z)=K_\eps(|z|)$ and which asymptotically have a
  singularity at the origin as $\eps\rightarrow 0$ (see
  \eqref{eq-molli-1}--\eqref{eq-molli-3} below). Here, $\Td$ is the flat
    torus with sidelength $\ell>0$, Furthermore, $M_\eps$ is defined by
\begin{align}
  M_\eps \ := \ \ome_{d-1} \int_0^1  K_\eps(t) t^{d-1} \ dt,
\end{align}
for $\eps>0$ and where $\ome_{d}$ is the volume of the unit ball in $\R^d$. In
this paper we consider the corresponding second order expansion for the
nonlocal perimeter functional, i.e. we consider the family of energies
\begin{align}\label{E-eps}
  \EE_\eps(u) \ %
  &:= \ M_\eps \dashint_{\Td} |\nabla u| \ dx   - \int_{\Rd} K_\eps(z) \dashint_{\Td} \frac{|u(x+z)-u(x)|}{2|z|} \ dxdz.
\end{align}
We assume that the kernels satisfy
$K_\eps \nearrow K_0$ for $\eps\searrow 0$ with
\begin{align} 
  \int_0^\infty K_0(r) r^{d-1} \ dr \ 
  = \ \infty \qquad 
      \label{eq-molli-1}     \\
  \int_\delta^\infty K_0(r)  \min \{ r^{d-2}, r^{d-1} \} \ dr \
  &< \ \infty 
  \qquad \FA{\delta > 0}.   \label{eq-molli-3}
\end{align}
In particular, this includes the following kernels
\begin{align+}
  K_\eps^{(1)}(r) \ &= \ r^{-d} \chi_{(\eps,\infty)}(r)  \label{exm-E1} \\
  K_{\eps}^{(2)}(r) \ &= \ r^{-d+\eps} \label{exm-E2} \\
  K_\eps^{(3)}(r) \ &= \ r^{-q} \chi_{(\eps,\infty)}(r) \qquad\qquad \text{for  $q > d$.} \label{exm-E3}
\end{align+}
We denote the corresponding energies by $\EE_\eps^{(i)}$ for $i = 1,2,3$.
When $d = 2$, $\EE_\eps^{(1)}$ is a sharp interface version of a problem from
micromagnetism for thin plates with perpendicular anisotropy
\cite{KB15,KMN19,nolte_phd, MS19}. In this setting, $\{u=1\}$ and $\{u=0\}$ are
the magnetic domains where the magnetization points either upward or
downward. The interfacial energy penalizes interfaces between the magnetic
domains, while the nonlocal term describes the dipolar magnetostatic interaction
energy between the domains. 
For $\EE_\eps^{(2)}$, the nonlocal term in the energy \eqref{E-eps} is the
so-called fractional perimeter $P_s(\Ome)$ with $s=1-\eps$ of the set
$\Ome:=\spt u$ (see
  e.g. \cite{CRS10, CV11,CabreCintiSerra-2020}), and the functional
is the second order asymptotic expansion of the fractional perimeter as
$s\nearrow 1$.  The third energy $\EE_\eps^{(3)}$ shows that our model also
includes some cases where the nonlocal term in the energy is more singular in
terms of its scaling than the perimeter.

\medskip

In the first two cases \eqref{exm-E1}--\eqref{exm-E2}, the nonlocal term and the
perimeter term in \eqref{E-eps} both formally have the same scaling for
$\eps = 0$. However, due to the non-integrability of the kernel $r^{-d}$ we have
$M_\eps \to \infty$, and both terms in \eqref{E-eps} are infinite for
$\eps = 0$. In our analysis we show that the singularity of these two terms
cancels in the leading order, and we calculate the remainder which describes
  the limiting behaviour of \eqref{E-eps} for $\eps \to 0$.  Our main
  result establishes compactness and $\Gamma$-convergence for the family of
  energies \eqref{E-eps}:
\begin{theorem}\label{thm-gamma} 
  Let $K_\eps \nearrow K_0$ where $K_0$ satisfies
  \eqref{eq-molli-1}--\eqref{eq-molli-3}.
 \begin{enumerate} 
\item \emph{(Compactness)}  For any family
  $u_\eps\in BV(\Td;\{0,1\})$ with $\sup_\eps \EE_\eps(u_\eps)<\infty$, there
  exists $u\in BV(\Td;\{0,1\})$ and a subsequence (not relabelled) such that
    \begin{align}
      u_\eps \to u \text{ in } L^1(\Td) \quad 
      \text{ and } \quad 
      \int_\Td |\nabla u_\eps| \ \to \ \int_\Td |\nabla u| \qquad \text{for $\eps \to 0$}.
    \end{align}
    \item \emph{($\Gamma$-convergence)} $\EE_\eps \Gamto \EE_0$ in the
  $L^1$--topology, where
  \begin{align} \label{E-0} %
    \EE_0(u) := \frac{1}{2}\int_{\Rd} \dashint_{\Td} K_0(|z|) \Big(\big
    |\frac{z}{|z|}\cdot \nabla u(x)\big| \chi_{B_1}(z)-
    \frac{|u(x+z)-u(x)|}{|z|} \Big) dx dz
\end{align}
for $u\in BV(\Td;\{0,1\})$ and $\EE_0(u)=+\infty$ otherwise.
\end{enumerate}
\end{theorem}
In fact, statement (i) of Theorem \ref{thm-gamma} is a consequence of the fact
that we have a uniform $BV$ estimate for our family of energies \eqref{E-eps} as
stated in Proposition \ref{prp-gamma}. The $\Gamma$-convergence result
demonstrates that the functional $\EE_0$, given in \eqref{E-0}, is the
second order of the asymptotic expansion for the nonlocal perimeter functional
and quantifies asymptotically the error in the approximation
\eqref{first-order}.

\medskip

A key observation for our analysis is that both interfacial energy and nonlocal
energy can be expressed solely in terms of the 
autocorrelation function 
\begin{align} \notag %
  C_u(z) \ := \ 
  \dashint_{\Td} u(x+z) u(x)\ dx .
\end{align}
As a result the energies $\EE_\eps$ for $\eps\geq 0$ can be written in terms of the autocorrelation functions
\begin{align}\label{eq:auto_energy}
\EE_\eps(u)=\int_{\R^d} \frac{C_u(z)-C_u(0)-z\cdot \nabla C_u(0)\chi_{B_1}(z)}{|z|} K_\eps(z) \ dz
\end{align}
for $\eps\geq 0$, cf. Proposition \ref{prp-energy-auto} for the radially
symmetrized expression. We note that the new formulated energy has a simpler
structure: it is linear in the space of autocorrelation functions. This leads to
easier proofs for the $\Gamma$-convergence for more general kernels compared
with the existing literature, e.g. \cite{MS19, CN20}. We note that the method of
autocorrelation function can also be used for a simpler proof of D\'avila's
result \cite{Da02}, see Appendix \ref{sec:davila}. We have chosen a periodic
  setting, but our methods can also be used to derive corresponding results for
  the corresponding full space problem.

\medskip

While autocorrelation functions are natural and often used tools in physics and
stochastic geometry, they seem not have been used in the context of nonlocal
isoperimetric problems before. We note that the idea of linearization of the
problem by a change of configuration space also appears in the formulation of
quantum field theory as well as in the theory of optimal transport. In both
cases the problem gets linear but the configuration space becomes very complex.

\medskip

We next consider the minimization problem for energies $\EE_\eps$ with
prescribed volume fraction, i.e.  for some fixed $\lam \in [0,1]$ we assume
\begin{align} \label{mass-constraint} 
  \dashint_\Td u \ dx \ = \ \lam.
\end{align}
We first note that for any $\eps \geq 0$, the minimization problem has a
solution: 
\begin{proposition}[Existence of minimizers] \text{}
  \label{prp-minex}
  For each $\lam \in [0,1]$ and each $\eps \geq 0$ there exists a
    minimizer $u$ of $\EE_\eps$ in the class of functions $u\in BV(\Td;\{0,1\})$
    which satisfy \eqref{mass-constraint}.
  \end{proposition}
The proof for Proposition \ref{prp-minex} is based on the uniform $BV$--bound in
    Proposition \ref{prp-gamma}. Next we note that the energy $\EE_\eps(\chi_\Ome)$
  only depends on the boundary $\p \Ome$  for any set $\Ome \SUS
  \Rd$. In fact, this can be seen from 
  \begin{align} \label{E-is-sym}%
    \EE_\eps(u) \ = \ \EE_\eps(1-u) \qquad \FA{\eps \geq 0,}
  \end{align}
 which follows from the symmetry property for the autocorrelation function
  (Proposition \ref{prp-auto1}\ref{it-Cu-complement}). In particular, the minimal energy for prescribed volume fraction is symmetric
  with respect to $\lam = \frac 12$ $\FA{\eps \geq 0}$. We thus expect that minimizers for $\EE_\eps$ in the case of the equal volume fraction $\lambda=\frac 12$ are equally distributed stripes.
  
  \medskip
  
  In the following, we discuss properties of minimizers of the limit energy for
  the special choice \eqref{exm-E1} or \eqref{exm-E2} of the kernel, i.e.
  $K_0(r) = r^{-d}$, and the corresponding energy is denoted by
  $\EE_0^{(1)}$. We first prove that when the volume is sufficiently small, the
  unique minimizer (up to translation) is given by a single ball:
  \begin{theorem}[Ball as minimizer for small mass] \text{}
    \label{thm-ball}
 Let $K_0(r)=r^{-d}$. Then there exists
      $m_0=m_0(d)>0$, such that if
\begin{align}
      \lam \ \leq \ \min \Big \{ \frac{m_0}{|\Td|}, \frac{\ome_d}{4^d} \Big \},
  \end{align}
  the unique, up to translation, minimizer $u$ of $\EE_0^{(1)}$ in
  $BV(\Td;\{0,1\})$ with constraint \eqref{mass-constraint} is a single ball.
\end{theorem}

Theorem \ref{thm-ball} actually holds for a slightly larger class of kernels,
cf. Proposition \ref{prp-ball}. The proof relies on the sharp quantitative
isoperimetric inequality in \cite{FMP} and a careful study of the limit energy
using the autocorrelation function. The optimality of the ball as minimizers has
been shown for related models with interfacial term and competing nonlocal
energy such as the Ohta-Kawasaki energy
\cite{KnuMur1,KnuMur2,Jul,BonCri,FFMMM,MurZal}. The difference in our model
in comparison to these previous results is that the nonlocal energy has
the same critical scaling as the perimeter, which requires a particularly
careful analysis. We note that Theorem \ref{thm-ball} is concerned with the case
of small volume instead of small volume fraction, as the smallness assumption
for $\lambda$ is not uniform in $\ell$.

  \medskip

To study the minimizers for small volume fraction, we calculate the energy
    $\EE_0^{(1)}$ for periodic stripes and lattice balls.  For this, we allow
    that the centers of the balls are arranged in an arbitrary Bravais lattice
    which requires a slightly more general formulation of our energy
    (cf. Definition \ref{defi-full}). In particular, we show that in $2d$ when
    the volume fraction is almost equal, equally distributed stripes have
    strictly smaller energy than all lattice ball configurations:
  \begin{proposition}[Ball vs. stripe patterns] %
   Let $d = 2$ and $K_0(r)=r^{-2}$. Then there is $c_0 > 0$ such
      that if $|\lam - \frac 12|<c_0$, then the energy of equally distributed
      stripes is strictly smaller than the energy of periodic ball configurations.
  \end{proposition}
  The precise statement and the proofs are given in Section
  \ref{sec-stripes-balls}.

\medskip

\textbf{Previous literature and related models.} We note that isoperimetric
problems have been extensively researched in the mathematical community. A
prototype model which has been investigated is the sharp--interface
Ohta-Kawasaki energy, where the system is described by a sharp interface term
together with a nonlocal Coulomb interaction term (or more general Riesz
interaction term), see e.g.
\cite{AceFusMor,AlbertiChoksiOtto-2009,ChoksiPeletier-2010,ChoksiPeletier-2011,CicSpa,Cri15,GolMurSerI,GolMurSerII,JulPis,MorSte}. However,
our model is different from these models, as the nonlocal term asymptotically
has the same (or higher) scaling as the interfacial energy.

\medskip

When $u\in BV(\R^2; \{0,1\})$ and $K_\eps$ is defined in \eqref{exm-E1}, the energy has
been studied by Muratov and Simon in \cite{MS19}, where the authors show among
other results the $\Gam$-convergence of the energy. The $\Gam$-convergence for the
energies \eqref{exm-E2} is studied by Cesaroni and Novaga \cite{CN20}.  Our
model includes the models considered in both \cite{MS19} and \cite{CN20} (in the
periodic setting). We generalize some results of these papers, however, with a
different strategy of proofs and in a more general setting. Moreover, we
characterize the minimizers of the limit energy when the volume is small, and
show that in $2d$ stripes have strictly smaller
energy than lattice of balls if the volume fraction is close to $1/2$. These
have not been considered in \cite{MS19, CN20}.

\medskip



  \textbf{Structure of paper.} In Section \ref{sec-prelim} we introduce the
  autocorrelation functions, collect and prove some of their basic properties. We write the energy in terms of the autocorrelation function and use it to explore some properties of the energy. We also
  derive different formulations of the energy. In Section \ref{sec-proofs}, we
  give the proof of Theorem \ref{thm-gamma} and Theorem \ref{thm-ball}. The stripes and balls configurations are studied in Section \ref{sec-stripes-balls}. 
  
\medskip

\textbf{Notation.} Throughout the paper unless specified we denote by $C$ a
positive constant depending only on $d$.  By $\Td :=\Rd/(\ell \Z)^d$ we denote
the flat torus in $\Rd$ with side length $\ell > 0$. We repeatedly identify
functions on $\Td$ with $\Td$--periodic functions on $\R^d$ using the canonical
projection $\Pi : \R^d \to \Td$. Similarly, any set $\Ome \SUS \Td$ can be
identified with its periodic extension onto $\R^d$.

\medskip

Let $\mathcal{M}$ be the space of signed Radon measures of bounded variation on
$\Td$.  For $\mu \in \MM$ we write $\|\mu\|:=|\mu|(\Td)$ for its total
variation. For any function of bounded variation $u\in BV(\Td)$, we analogously
write $\NN{\nabla u}=|\nabla u|(\Td)$. By the structure theorem
for BV functions we have $\nabla u=\sigma |\nabla u|$ for some
$\sigma:\Td\to \S^{d-1}$. By a slight abuse of notation we sometimes write
  \begin{align}
    \int_{\Rd} \big|\ome \cdot \nabla u(x)\big|  \ dx\ %
    := \ \int_{\Rd} \big|\ome \cdot \sig(x)\big| d|\nabla u|. \ %
  \end{align}
  Given a measurable set $\Ome\subset \Td$
(or $\Ome \SUS \Rd$), the Lebesgue measure of $\Omega$ is denoted by $|\Ome|$. If $\Ome$ has
finite perimeter, the perimeter of $\Ome$ is denoted by
$P(\Ome) := \NN{\nabla\chi_{\Ome}}$. Furthermore, $\ome_{d}$ is the volume of
the unit ball in $\Rd$ and $\sig_{d-1} = d\ome_d$ is the area of the
$(d-1)$--dimensional unit sphere $\S^{d-1} \SUS \R^d$.

\section{Autocorrelation functions and energy} \label{sec-prelim}

In this section we give various different formulations for our energy, in
particular in terms of the autocorrelation function.

\subsection{Autocorrelation functions}

In this section we introduce the autocorrelation function in our setting and show some of properties of the autocorrelation function. At
the end of the section, we compute explicitly the autocorrelation function for a single ball and stripes. The autocorrelation function of a
characteristic function is defined as follows:
  \begin{definition}[Autocorrelation functions]\label{def:correlation}
    Let $\ell > 0$. For $u\in L^1(\Td, \{0,1\})$, we define the autocorrelation
    function $C_u: \Rd\to [0,\infty)$ by
  \begin{align} \label{def-Cu} %
    C_u(z) \ := \ \dashint_{\Td} u(x+z)u(x)\ dx.
  \end{align}
  Its radially symmetrized version $c_u:[0,\infty)\to [0,\infty)$ is given by
  \begin{align} \label{def-cu} %
    c_u(r) \ := \ \dashint_{\p B_r} C_u(y)\ dy \ = \ \dashint_{\mathbb{S}^{d-1}} C_u(r\ome)\ d\ome
  \end{align}
\end{definition}
We can also write the autocorrelation function \eqref{def-Cu} as
$C_u =\frac{1}{|\Td|} u* \mathcal{I}u$, where $\mathcal{I}u(x):=u(-x)$.  We note that both $C_u$
and $c_u$ are dimensionless. Moreover, they are invariant under translation and
reflection with respect to the origin.

\medskip

Autocorrelation functions are widely used in other fields such as physics and
stochastic geometry: One common use of these functions is to study spectral
properties of observed patterns in theoretical physics (e.g. astrophysics).  In
the field of stochastic geometry, the autocorrelation function is also called
covariogram or set covariance (cf. e.g. \cite{Matheron75} and \cite{AB09}). A
common question in this field is to reconstruct a set from its autocorrelation
function.  One conjecture is e.g. whether the autocorrelation function for a
convex set can characterize the set uniquely up to translation and reflection
\cite{Matheron75}.

\medskip

We note that the definition of autocorrelation function can be extended to more
general configurations:
\begin{remark}[General periodic configurations] \label{rmk-general-pf} %
  Let $\Lam \SUS \Rd$ be a periodicity cell and assume that $u : \Rd \to \R$ is
  $\Lam$--periodic, i.e. $u(x + \Lam) = u(x)$. Then the autocorrelation function
  $C_u$ can be defined by
\begin{align}\label{cu-general} 
  C_u(z) \ = \ \lim_{\ell\to \infty}\dashint_{[-\ell,\ell]^d} u(x+z)u(x) \ dx.
\end{align}
Alternatively, one can define it by replacing $\Td$ by the periodicity cell in
\eqref{def-Cu}, and it is easy to see that the two definitions
coincide. In fact, the definition \eqref{cu-general} is also well--defined for
more general functions such as almost periodic functions,
cf. eg. \cite{Corduneanu09}. The symmetrized autocorrelation function is defined
accordingly. Equation \eqref{cu-general} shows in particular that $C_u$ (and
$c_u$) do not depend on the choice of periodicity cell.
\end{remark}

\medskip

Some properties of the autocorrelation function are derived in e.g.
\cite[Prop. 11]{Galerne11}. However, otherwise we have not found many results
about analytical properties of the autocorrelation function in the mathematical
literature. We hence collect some basic properties about $C_u$ in the following
proposition:
\begin{proposition}[Properties of $C_u$] \label{prp-auto1} %
  Let $u, \t u \in L^1(\Td;\{0,1\})$. Then
  $C_u \in C^0(\Rd,[0,\infty))$, $C_u$ is $\Td$-periodic and
    $C_u(-z) = C_u(z)$. Furthermore, 
    \begin{enumerate}      
  \item\label{it-Cu-sup} $\DS C_u(0) \ = \ \NPL{C_u}{\infty}{\Rd} \ = \ \frac 1{|\Td|} \NPL{u}{1}{\Td}$
  \item\label{it-Cu-ucon}
    $\DS \NPL{C_u - C_{\t u}}{\infty}{\Rd} \ \leq \ \frac 1{|\Td|} \NPL{\t u -u}{1}{\Td}$.
  \item\label{it-Cu-one} $\DS \NPL{C_u}{1}{\Td} \ = \ \frac 1{|\Td|} \|u\|^2_{L^1(\Td)}$.
  \item\label{it-Cu-lip} $\DS |C_u(z_2)-C_u(z_1)| \ \leq \ C_u(0)-C_u(z_2-z_1)$
    \qquad %
    \FA{z_1,z_2\in \Rd.}
    \item\label{it-Cu-complement} $C_{1-u} = C_u + 1 - 2 \NPL{u}{1}{\Td}$.
  \end{enumerate}
  If $u\in BV(\Td;\{0,1\})$ then $C_u \in C^{0,1}(\Rd)$. Furthermore,
  \begin{enumerate}[resume]
  \item\label{it-Cu-BV}  $\DS \NPL{\p_iC_u}{\infty}{\Rd} \leq \ \frac 1{|\Td|} \NN{\p_i u}$,\quad\qquad \FA{i\in \{1,\cdots, d\}}.
  \item\label{it-Cu-BV2}
    $\p_{ij} C_u = - \frac 1{|\Td|} \p_i u* \mathcal{I}\p_j u \in \mathcal{M}$
    where $\mathcal{I}u(x):=u(-x)$ and
    \begin{align}
      \NN{\p_{ij} C_u} \ \leq \ \frac 1{|\Td|} \NN{\p_i u} \NN{\p_ju}
    \quad\qquad \FA{i,j\in \{1,\cdots, d\}}.
    \end{align}
  \end{enumerate}
\end{proposition}
\begin{proof}
  The periodicity
  and symmetry of $C_u$ follows directly from the periodicity of $u$ and the
  change of variable $y=x-z$.

  \medskip
  
  \textit{\ref{it-Cu-sup}--\ref{it-Cu-lip}:} For \textit{\ref{it-Cu-sup}} we
  note that $|\Td|C_u(z)=|\Ome^{per}\cap (\Ome-z)|$ for
  $\Ome:=\spt u\subset \R^d$ and $\Ome^{per}=\cup_{e\in (\ell
    \Z)^d}(\Ome+e)$. Hence, clearly $C_u \in C^0(\Rd,[0,\infty))$ and takes
  the maximum at $z=0$. By the triangle inequality and since
  $\NP{\t u}{\infty}, \NP{u}{\infty} \leq 1$, for any $z\in \Rd$ we have
  \begin{align}\label{eq-Cu-ucon}
    |C_{u}(z)-C_{\t u}(z)| \ %
    &\leq \ \frac 1{|\Td|} \Big(\NPL{(u-\t u)(\cdot+z)}{1}{\Td} + \NPL{u-\t u}{1}{\Td} \Big) \\ %
    &= \ \frac 2{|\Td|} \NPL{u-\t u}{1}{\Td}.
  \end{align}

  The identity \textit{\ref{it-Cu-one}} follows directly from Fubini and since
  $u$ is periodic. Estimate \textit{\ref{it-Cu-lip}} is given in
  \cite[Prop.5]{Galerne11}.

  \medskip

  \textit{\ref{it-Cu-BV}:} By periodicity and since
  $|u(x+z)-u(x)|=|u(x+z)-u(x)|^2$ we get
    \begin{align}\label{eq-equi-0}
      C_u(0)-C_u(z) \ %
      &= \ \dashint_{\Td} u(x)^2-u(x)u(x+z) dx \\ %
     &= \ \frac{1}{2}\dashint_{\Td}|u(x+z)-u(x)| dx \ %
       \leq \ \frac{\NNN{\nabla u}{} |z|}{2|\Td|} \qquad %
       \FA{z\in \R^d.}
  \end{align}
  Together with \textit{\ref{it-Cu-lip}}, this yields \textit{\ref{it-Cu-BV}}.
  
  \medskip
  
  \textit{\ref{it-Cu-BV2}:} %
  We choose a sequence $u_\eps \in \cciL{\Td}$ with
  $\nabla u_\eps \wtos \nabla u$ in $\MM$.  For $u_\eps\in C^\infty_c(\Td)$ and
  with the notation $(\mathcal I f)(x) := f(-x)$ we have
  \begin{equation}\label{C''}
    \begin{split}
      \p_{ij} C_{u_\eps}(z)\ %
      &= \ \p_i \dashint_{\Td} \p_j u_\eps (x+z)u_\eps(x) \ dx \ %
      = \ \p_i\dashint_{\Td} \p_j u_\eps(y) u_\eps(y-z) \ dy\\
      &=- \dashint_{\Td} \p_j u_\eps (y) \p_i u_\eps (y-z)\ dy=- \frac 1{|\Td|} (\p_j u_\eps * \mathcal{I}\p_i u_\eps)(z).
    \end{split}
  \end{equation}
  We note that \eqref{C''} with $g := \p_i u_\eps$ and
  $f := \p_j u_\eps$ defines a bilinear operator $\MM \times \MM \to \MM$
  with $(g,f) \mapsto g * \II f$. Since
  $\NNN{\mu_1 * \mu_2}{\MM} \leq \NNN{\mu_1}{\MM} \NNN{\mu_2}{\MM}$, we get
  $\phi_\eps * \psi_\eps \wtos \phi * \psi$ for all
  $\phi_\eps, \psi_\eps \in \MM$ with $\phi_\eps \wtos \phi$,
  $\psi_\eps \wtos \psi$ in $\MM$. Passing to the limit in the weak formulation
  of \eqref{C''} one gets $\p_{ij}C_u\in \mathcal{M}$ and moreover
  $|\Td|\p_{ij}C_u=-\p_j u \ast \mathcal{I}\p_i u$.

  \medskip

  \textit{\ref{it-Cu-complement}:} Follows directly from the definition.
  \end{proof}
  We note some of the assertions in Proposition \ref{prp-auto1} can be
    easily generalized to the case of $u \in BV(\Td)$ in which case an
    additional norm $\NI{u}$ has to be added on the right hand side.

  \medskip

  In the following we derive properties of the symmetrized
  autocorrelation. Similarly, these properties extend to general functions
  $u \in BV(\Td)$. However, in the general case we have to replace $\NN{\nabla u}$ by
  $\NN{D_s u}$ in \ref{itc-perim} where $D_s u$ is the jump part of the
  derivative of $u$.
  \begin{proposition}[Properties of $c_u$]\label{prp-auto2} %
  Let $u, \t u \in L^1(\Td;\{0,1\})$. Then $c_u \in C^0([0,\infty))$ and $c_u$ is
  uniformly continuous. Furthermore, 
  \begin{enumerate}
  \item\label{itc-mass}  %
$ \displaystyle c_u(0) \ = \ \NNN{c_u}{L^\infty([0,\infty))} \ = \ \frac 1{|\Td|} \NPL{u}{1}{\Td}$.
\item \label{itc-ucon} $\DS \|c_u-c_{\t u}\|_{L^\infty([0,\infty))} \ %
  \leq \ \frac{1}{|\Td|}\|u-\t u\|_{L^1(\Td)}$.
\item\label{itc-diff} %
  $|c_u(r_2)-c_u(r_1)| \ %
  \leq \ c_u(0)-c_u(r_2-r_1) \qquad\qquad \FA{0\leq r_1\leq r_2<\infty}$.
  \item\label{itc-diffquot} %
    $\displaystyle \frac 1r (c_u(0)-c_u(r)) \ %
    = \ \frac{1}{2} \dashint_{\S^{d-1}} \dashint_{\Td} \frac 1r \left|
      u(x)-u(x+r\ome)\right| \ dx d\ome$ \quad \FA{r > 0}.
  \end{enumerate}
  If $u \in BV(\Td;\{ 0, 1 \})$ then $c'_u\in BV_{loc}((0,\infty))$ and
  $c'_u(0)$ exists.  Furthermore,
  \begin{enumerate}[resume]
  \item\label{itc-perim} %
    $\DS c'_u(0) \ = \ - \NNN{c_u'}{L^\infty([0,\infty))} \  \ %
   = \ - \frac{\ome_{d-1}}{\sig_{d-1}}\frac{\NN{\nabla u}}{|\Td|}$.
  \item\label{itc-asym} (Asymptotics as $r\rightarrow \infty$) When $d\geq 2$ and $r\ell\geq 1$ we have 
       \begin{align}
       \frac{1}{\ell}|c_u(r) - c_u(0)^2|  + |c'_u(r)|\leq \frac{C}{ r^{\frac{d-1}{2}}\ell^{\frac{3d-1}{2}}}\NNN{ \nabla u}{}^2.
       \end{align}
     \item\label{itc-bvloc} Let $\Ome=\spt u$ with outer unit normal $\nu$. Then
       for each $\phi\in C_c^\infty((0,\infty))$ and with $z := y-x$ we have
   \begin{align}
     \int_{0}^\infty c'_u(r)\phi'(r)\ dr \
     &= \ \frac 1{\sig_{d-1}|\Td|} \int_{\p \Ome}\int_{\p \Ome} \frac {\big(\nu(y)\cdot \frac{z}{|z|}\big)\big(\nu(x)\cdot \frac{z}{|z|}\big)}{|z|^{d-1}}  \phi(|z|) \ dy dx.
   \end{align}
  \end{enumerate}
\end{proposition}
\begin{proof}
  \textit{\ref{itc-mass}--\ref{itc-diff}:} Follows from Proposition
  \ref{prp-auto1}(i) by taking the spherical mean.

  \medskip

  \textit{\ref{itc-diffquot}:} Taking the spherical average of \eqref{eq-equi-0}
  and by \ref{itc-mass}, we get \ref{itc-diffquot}.

  \medskip
  
  \textit{\ref{itc-perim}:} %
  By \ref{itc-diffquot} and \cite[Prop. 11]{Galerne11} (noting that the identity
  also holds for periodic functions) the derivative $c'(0)$ exists and we have
  \begin{align} \label{ab-ineq} 
    - c'_{u}(0)\ %
    &\stackrel{\ref{itc-diffquot}}= \ \lim_{r\to 0+} \dashint_{\mathbb{S}^{d-1}} \dashint_{\Td} \frac 1{r}  |u(x)-u(x+r\ome)| \ dx d\ome\\
    &= \ \frac{1}{2}  \dashint_{\Td} \dashint_{\mathbb{S}^{d-1}} |\omega\cdot \nabla u(x)| \  d\ome dx \ %
    = \ \frac{\ome_{d-1}}{\sig_{d-1}}\frac{\NN{\nabla u}}{|\Td|}. %
  \end{align}
  For the last identity, we have used that
  \begin{align}\label{eq-identity}
    \int_{\mathbb{S}^{d-1}} |\omega\cdot e_1| \  d\ome \ %
    = \ \frac{2\sig_{d-2}}{d-1} \ = \ 2\omega_{d-1}.
  \end{align}
  \DET{%
    With $r = |e_1\cdot \ome|$ and with cylindrical coordinates we have(??)
    \begin{align}
      \int_{\S^{d-1}} |e_1 \cdot \ome| \ d\ome %
      = \  2 \sig_{d-2} \int_0^1 r (1-r^2)^{\frac {d-2}2} \ dr \ %
      =  \   \frac {2 \sig_{d-2}}d
      =  \   2 \ome_{d-1} \frac {d-1}d 
    \end{align}
  I think some jacobian factor is missing. One can directly use the spherical coordinates: $e_1\cdot \ome=\cos\theta$ and $d\ome_{d-1}=(\sin\theta)^{d-2} d\theta d \omega_{d-2}$ to get 
  \begin{align}
    \int_{\S^{d-1}} |e_1\cdot \ome| d\ome \ %
    = \ \sig_{d-2}\int_0^\pi |\cos\theta| (\sin\theta)^{d-2} d\theta\ %
    = \ \frac{2\sig_{d-2}}{d-1}.
  \end{align} 
  If we take the spherical average of $c_u$, then we can also call this integral
  $\gamma_d$ instead of calculating it out.}
To conclude we note that
by \ref{itc-diff} we have $\NNN{c_u'}{L^\infty([0,\infty))} \ \leq \ -c'_u(0)$.

  \medskip

  \textit{\ref{itc-asym}:} For $k \in (\frac{2\pi}{\ell}\Z)^d$, let $\alp_k$ be
  the Fourier coefficient of $C_u$. Since $C_u(x)=C_u(-x) \in \R$ \FA{x \in
    \Td}, we have $\alp_{-k} = \alp_k \in \R$ and
    \begin{align}
      \alp_k \ %
      &:= \ \dashint_{\Td} C_u(x) e^{-i k \cdot x} \ dx
        \ = \ \dashint_{\Td} C_u(x) \cos (k \cdot x) \ dx.
    \end{align} 
    Then
  \begin{align}
    C_u(x) \ &= \ \sum_{k \in (\frac{2\pi}{\ell}\Z)^d} \alpha_k e^{i k \cdot x} %
               \ = \sum_{k \in (\frac{2\pi}{\ell}\Z)^d} \alpha_k \cos(k \cdot x)      \label{F-series}.
    \end{align}
    Integrating by parts and since
    $\NN{D^2 C_u} \leq \frac 1{|\Td|}\NN{\nabla u}^2$
    (Prop. \ref{prp-auto1}\ref{it-Cu-BV2}), we get
    \begin{align} \label{alp-dec} %
      |\Td|^2 k^2 |\alp_k| \ \leq \ \NNN{\nabla u}{}^2
    \end{align}
    Thus the series \eqref{F-series} is absolutely convergent.  Then
    \begin{align}
      c_u(r) \ %
      &= \ \sum_{k} \alp_k \dashint_{\S^{d-1}} e^{i k \cdot r \ome} \ d\omega \ %
        = \sum_{k} \alp_k  \int_{0}^{\pi} \cos(|k| r \cos \theta) (\sin\theta)^{d-2}\ d\theta, \ %
    \end{align}
    since $\int_0^\pi \sin(|k|r\cos\theta)(\sin\theta)^{d-2} d\theta=0$.
    Note that $\alp_0 = c(0)^2$ and thus
    \begin{align}
      c_u(r) - c_u(0)^2 \  %
      = \ (2\pi )^{\frac d2}\sum_{k \neq 0} \alp_k J_{\frac{d-2}{2}} (r|k|) (r|k|)^{\frac{2-d}{2}}, \label{series-con}
    \end{align}
    where $J_\alpha$ are the Bessel functions of the first kind. The decay
    estimate for $|c_u(r)-c_u(0)^2|$ then follows from the estimate
    $|J_{\alp}(t) |\leq \sqrt{\frac 2{\pi t}} + \OO(\frac 1t)$ for $\alp \geq 0$
      and $t \geq 1$ (cf. \cite[10.17]{DLMF}).

      \medskip
      
      Using the recurrence relation
      $\frac{d}{dt}(t^{-\alp}J_{\alp}(t)) = -t^{-\alp} J_{\alp+1}(t)$ we obtain
      (the series in \eqref{series-con} is absolutely and uniformly
        convergent for all $r > 0$ by \eqref{alp-dec})
    \begin{align} \label{series-dercon} %
      c'_u(r) \ %
       &\lupref{series-con}= \ -(2\pi)^{\frac d2} \sum_{k \neq 0} \frac{|k|\alp_k}{(r|k|)^{\frac{d-2}{2}}} J_{\frac d2}(r|k|),
    \end{align}
    The decay estimate for \eqref{series-dercon} then follows analogously as for
    \eqref{series-con}.  This completes the proof for \ref{itc-asym}.  \DET{
      Using $\hk{|J_{\alp}(t) |\leq \sqrt{\frac 2{\pi t}} + \OO(\frac 1t)}$ and
      \eqref{alp-dec} and with $k = \frac {2\pi}{\ell} q$ we get
    \begin{align}
      |c'_u(r)| \ %
      &\leq \  \frac{(2\pi)^{\frac d2} \NNN{\nabla u}{}^2}{r^{\frac{d-1}{2}}|\Td|} \sum_{q \in \Z^d \BS \{ 0 \} } \frac 1{|k|^{\frac{d+1}{2}}} + \OO\Big(\frac 1{r^{\frac d2} } \sum_{q \in \Z^d \BS \{ 0 \} } \frac 1{|k|^{\frac{d+2}{2}}} \Big) \\ %
      &\leq \  \frac{\NNN{\nabla u}{}^2\ell^{\frac{d\hk{+}1}{2}}}{\pi r^{\frac{d-1}{2}} |\Td|} \Big( 1 + \frac C{(\ell r)^{\frac 12}}\Big) \sum_{q \in \Z^d \BS \{ 0 \} } \frac 1{|q|^{\frac{d+1}{2}}}  \\ %
      &\leq  \frac{1}{\pi r^{\frac{d-1}{2}} \ell^{\frac{d-1}{2}}} \Big( 1 + \frac C{(\ell r)^{\frac 12}}\Big) \NNN{ \nabla u}{}^2,
    \end{align}}
   
  \medskip
  
  \textit{\ref{itc-bvloc}:} For any test function
  $\phi\in C_c^\infty((0,\infty))$ we calculate
  \begin{align}
   \int_0^\infty c'_u(r) \phi'(r) \ dr   \ %
    &\ = \ - \frac 1{|\Td|} \int_0^\infty \dashint_{\mathbb{S}^{d-1}} \frac {d}{dr} \big[\ome \cdot \nabla C_u(r \ome)\big] \phi(r)\ d\ome dr \\ %
    &\ =\ -  \frac 1{\sig_{d-1}|\Td|} \int_{\Rd} \frac 1{|z|^{d+1}} \skp{z}{D^2 C_u(z) z} \phi(|z|)  \ dz,
  \end{align}
  where we have integrated by parts in the last line.  By Proposition \ref{prp-auto1}
\begin{align}
  \skp{z}{D^2 C_u z} \ = \ - \frac 1{|\Td|} (z \cdot \nabla u) * \mathcal{I}(z \cdot \nabla u) \ %
  \ \in\MM. 
\end{align}
Since $\nabla u =-\nu \HH^{d-1}\lfloor \p \Ome$ and with $z:=y-x$ we
then get
  \begin{align}
    \int_0^\infty c'_u(r) \phi'(r) \ dr  \ = \ \frac {1}{\sig_{d-1} |\Td|} \int_{\p \Ome}\int_{\p \Ome} \frac{ (\nu(y)\cdot z) (\nu(x)\cdot z)}{|z|^{d+1} }\phi(|z|)\ dydx.
  \end{align}
  For $\spt \phi \in [\d, \infty)$ the expression above is estimated by
  $C(\d) \NI{\phi}$. This shows that $c'_u\in BV_{loc}((0,\infty))$.
\end{proof}
\begin{figure}
  \centering %
  \includegraphics[width=7cm]{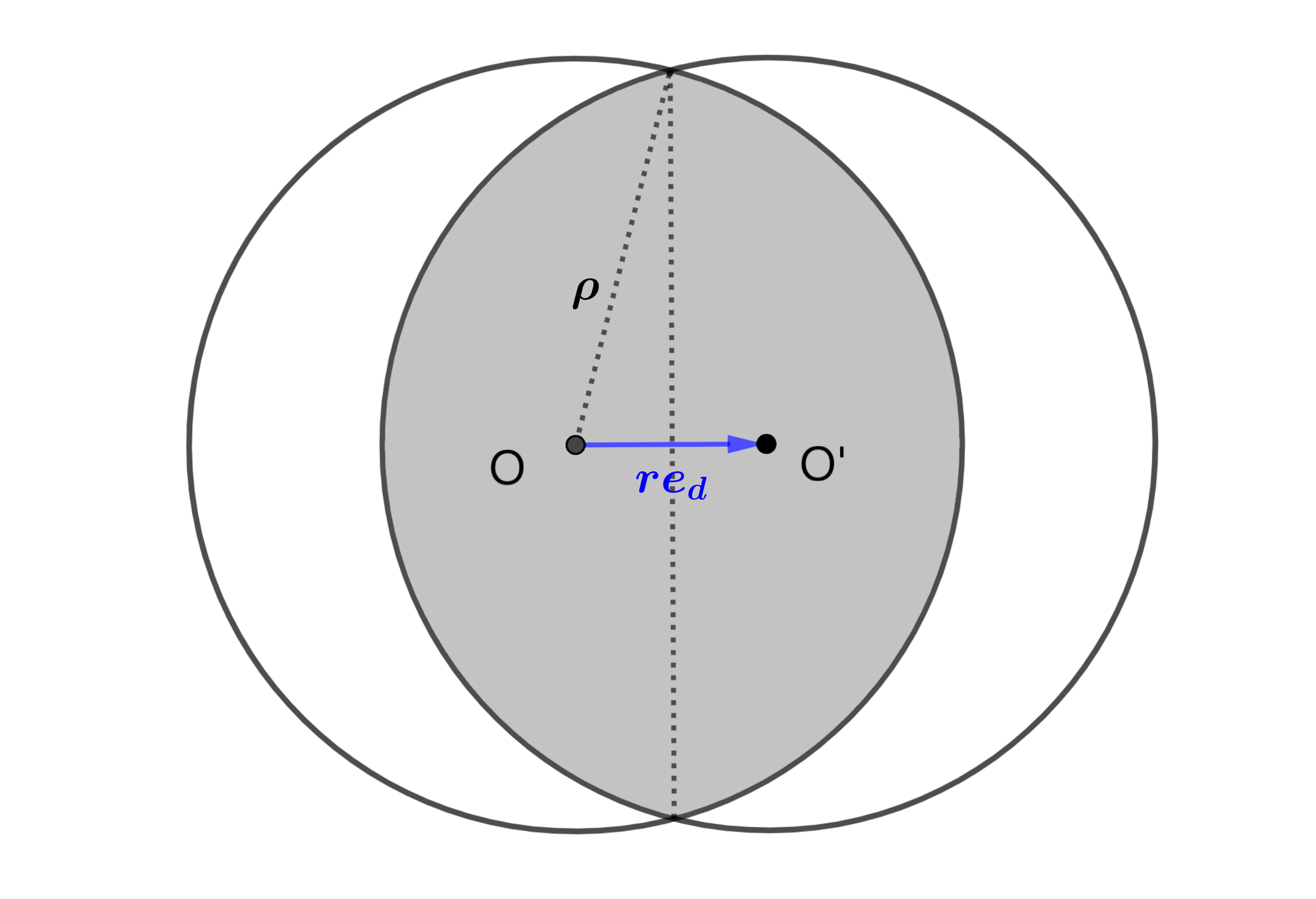}
  \caption{Illustration for autocorrelation function $C_u(r)$ for a ball.}
  \label{fig-ball}
\end{figure}
\begin{lemma}[Autocorrelation functions for balls and
  stripes] \label{lem-ball-cu} \text{} %
\begin{itemize}
\item [(i)] Let $u=\chi_{B_\rho}$ for some $\rho\in (0, \frac{\ell}{4})$. Then
  for $r\in [0,2\rho]$ we have
  \begin{align} %
    c_{u}(r) \ %
    &= \   \frac{2\ome_{d-1}}{|\Td|}  \Big(\int_{\frac r{2\rho}}^{1}  (1 - t^2)^{\frac {d-1}2} \ dt \Big) \rho^d, \label{cu-ball}
    \qquad %
  \end{align}  
\item [(ii)] Let $u=\chi_{S}$, where
  $S=\cup_{i=1}^k [a_i,b_i]\times [0,\ell]^{d-1}$ is a finite union of
  stripes. Let $d_0=\min_i|b_i-a_i|>0$. Then $c_u$ is affine linear near
    $0$, i.e.
  \begin{align}
    c_u(r) \ = \ \frac { \NPL{u}{1}{\Td}}{|\Td|} - \Big(\frac{\ome_{d-1}}{\sig_{d-1}}\frac{\NN{\nabla u}}{|\Td|} \Big)  r  \qquad %
    \text{ for } r\in (0,d_0).
  \end{align}
  \end{itemize}
\end{lemma}
\begin{proof}
  \textit{(i):} By invariance and using cylindrical coordinates we have
  \begin{align}
    c_{u}(r) \ %
    = \ \frac 1{|\Td|} |B_\rho\cap B_{\rho}(re_d)|
    = \ 2\ome_{d-1}  \frac{1}{|\Td|}\int_{\frac r2}^{\rho}  (\rho^2 - s^2)^{\frac {d-1}2} ds.
  \end{align}
  for $r \in [0, 2\rho]$ (see Fig. \ref{fig-ball}) and
  \begin{align}
      c'_{u}(r) \  %
     &=  \ -   \frac {\ome_{d-1}}{|\Td|} \Big( 1 - \frac{r^2}{4\rho^2} \Big)^{\frac{d-1}2}  \rho^{d-1}. \label{cus-ball}
  \end{align}

  \medskip
  
  \textit{(ii):} For given $S$ and $r\in (0,d_0)$ we have
  \begin{align}
      c_u(0)- c_u(r)=\frac{\ell^{d-1}}{|\Td|}\int_{\S^{d-1}}|r\ome\cdot e_1|\
      d\ome \stackrel{\eqref{eq-identity}}{=}
    \frac{2\ome_{d-1}\ell^{d-1}}{|\Td|}r=-c'_u(0)r.
  \end{align}
  Then the claim follows from Proposition \ref{prp-auto2}.
\end{proof}

\subsection{Properties of the energies}

In this section, we express our energies in terms of the autocorrelation function
and use this expression to explore their properties. We first note that by
Proposition \ref{prp-auto2} our energies can be expressed in terms of the
symmetrized autocorrelation function $c_u$:
\begin{proposition}[Energy in terms of $c_u$]\label{prp-energy-auto} %
  Let $d \geq 1$. Let $\eps \geq 0$ and suppose that $K_\eps, K_0$ satisfy
    \eqref{eq-molli-1}--\eqref{eq-molli-3}.  Then for any
  $u \in BV(\Td;\{ 0, 1 \})$ we have
  \begin{align}\label{eq-EE1-auto} %
    \EE_\eps(u) \ %
    &= \ \sig_{d-1} \int_0^{\infty} \left(c_u(r)-c_u(0)-r c'_u(0) \chi_{(0,1)}(r) \right) K_\eps(r) r^{d-2} \ dr.
  \end{align}
  Furthermore, there is $C=C(K_0,d)>0$ such that $\EE_\eps(u)\geq -C$ for
  $\eps\geq 0$.
\end{proposition}
\begin{proof}
  We first note that the integral in \eqref{eq-EE1-auto} is well-defined in
  $\R \cup \{ \pm \infty \}$: Indeed, in view of Proposition
  \ref{prp-auto2}(i)(ii) we have $c_u(r)-c_u(0)-rc'_u(0) \geq 0$ and
  $c_u(0)-c_u(r) \geq 0$ \FA{r > 0}. Thus the integrand is nonnegative in
  $(0,1)$ and negative in $(1,\infty)$. This combined with our assumptions
  \eqref{eq-molli-1}--\eqref{eq-molli-3} on the kernels $K_\eps$ and that
  $c_u\in C^{0,1}((0,\infty))$ gives that the integral \eqref{eq-EE1-auto} is
  well-defined in $\R$ for $\eps > 0$, and is well-defined in
  $\R \cup \{ + \infty \}$ for $\eps = 0$. To show the uniform lower bound we observe that
  \begin{align}
  \EE_\eps(u)&\ \geq \ -\sig_{d-1}\int_1^\infty (c_u(0)-c_u(r)) K_\eps (r) r^{d-2} \ dr \\
 & \ \geq \ -\sig_{d-1}\frac{\|u\|_{L^1(\Td)}}{|\Td|}\int_{1}^\infty K_\eps(r) r^{d-2} \ dr\geq -C,
  \end{align}
   where in the last inequality we have used the assumption \eqref{eq-molli-3}. In view of Proposition \ref{prp-auto2}(i)(ii) and using polar
  coordinates we have
    \begin{align}
      \EE_\eps(u) %
      &= \ - \sig_{d-1} c_u'(0)\int_0^1  K_\eps (r) r^{d-1} dr  \\ 
      &\qquad -  \int_0^\infty  \bigg[ \dashint_{\S^{d-1}} \frac{1}{2}\dashint_{\Td} |u(x+r\ome)-u(x)| dx \bigg]  K_\eps (r) r^{d-2}  d\ome dr.
    \end{align}
    Then \eqref{eq-EE1-auto} directly follows from Proposition
      \ref{prp-auto2}\ref{itc-diffquot}.
  \end{proof}
  The truncation at $r = 1$ is arbitrary. More generally, we have:
\begin{remark}[Different truncation] \label{rmk-truncation}%
  Let $\eps \geq 0$.  For any $r_0 > 0$ we have
  \begin{align} \label{trunc} %
    \EE_\eps(u)  \ %
    &= \ \sig_{d-1}\int_0^{\infty} \Big[c_u(r)-c_u(0)- rc'_u(0) \chi_{(0,r_0)}(r)   \Big] K_\eps(r) r^{d-2}\ dr \\
    &\qquad\qquad   +  \sig_{d-1}c_u'(0) \int_1^{r_0} K_\eps(r) r^{d-1} \ dr.
  \end{align}
\end{remark}
Although the energies are linear in terms of the autocorrelation functions, they are not linear in $u$. Indeed, they are sublinear in the following sense:
\begin{lemma}[Interaction energy]
  \label{lem-interaction}
  Let $\eps \geq 0$. Suppose that $u_1,u_2\in BV(\Td;\{0,1\})$ have
    essentially disjoint support, i.e. $u_1 + u_2 \in BV(\Td;\{0,1\})$ and 
  $\NN{D(u_1+u_2)}=\NN{Du_1}+\NN{Du_2}$. Then the interaction energy
  \begin{align}
  \II_\eps(u_1,u_2) \ %
    := \ \EE_{\eps}(u_1+u_2)-\EE_{\eps}(u_1)-\EE_{\eps}(u_2)
  \end{align}
  is non--negative and for $\Ome_i:=\spt u_i$, $i=1,2$ we have
 \begin{align}
   \II_\eps(u_1,u_2) \ %
   = \ \frac{2}{|\Td|}\sum_{e\in (\ell\Z)^d}\int_{\Ome_1+e}\int_{\Ome_2} \frac{K_\eps(|x-y|)}{|x-y|} dxdy .
 \end{align}
\end{lemma}
\begin{proof}
  Let $u := u_1 + u_2$. By assumption, we have $c_u(0)=c_{u_1}(0)+c_{u_2}(0)$
  and $c'_u(0)=c'_{u_1}(0)+c'_{u_2}(0)$. Together with $u=u_1+u_2$ the assertion
  follows from the energy expression \eqref{eq-EE1-auto} and
  \begin{align}
    \II_{\eps}(u_1,u_2) \ %
    &= \ \sig_{d-1}\int_0^\infty \big( c_{u_1+u_2}(r)-c_{u_1}(r)-c_{u_2}(r) \big) K_\eps(r) r^{d-2} \ dr \\
    &=\int_{\Rd}\int_{\Td}\Big( u(x)u(x+z)
                   - \sum_{i=1,2} u_i(x)u_i(x+z) \Big)  \frac{K_\eps(|x-y|)}{|x-y|} \ dxdz.
  \end{align}
  Using the symmetry of the sum and substituting the supports of  $u_i$, we obtain the desired expression for $\II_\eps$.
\end{proof}
The space of functions with finite limit energy is quite
complicated: the condition $u \in BV(\Td, \{ 0, 1 \})$ is in general only a necessary but not
sufficient condition for a configuration to have finite energy. Below we construct a BV function with infinite energy $\EE_0^{(1)}$ using the nonnegativity of the interaction energy in Lemma \ref{lem-interaction}:
\begin{remark}[$BV$--function with infinite energy $\EE^{(1)}_0$]\label{rmk-infinite_energy}
  Let $(B_{r_k}(x_k))_{k=2}^\infty$ be a sequence of disjoint balls in $\Td\subset \R^2$
  with radius $r_k=\frac{1}{k(\ln k)^2}$. Let
  $u:=\sum_{k=2}^\infty \chi_{B_{r_k}(x_k)}$. On one hand, it is immediate from
  $\sum_{k=2}^{\infty}r_k<\infty$ that $u\in BV(\Td;\{0,1\})$. On the other
  hand, from Lemma \ref{lem-interaction} and the estimate of the energy for a
  single disk (cf. Lemma \ref{lem-small_disk}(i) below) we have
  \begin{align}
    \EE^{(1)}_0(u) \ \geq \ \sum_{k=2}^\infty \EE_0^{(1)}(u_k)\geq C\sum_{k=2}^\infty (-\ln
    r_k)r_k \ = \ +\infty.
  \end{align}
\end{remark}

We end this section by a remark on the finite energy configurations for $\EE_0^{(3)}$:
\begin{remark}[Finite $\EE^{(3)}_0$ energy configurations]
  The family of sets which have finite energy $\EE_0^{(3)}$ becomes increasingly
  smaller as $q$ increases. Indeed, as $q$ getting larger the (symmetrized)
  autocorrelation function $c_u$ needs to converge faster to its affine
  approximation at $0$ in order to get a finite energy (cf. Proposition
  \ref{prp-energy-auto}). It follows from the explicit computation of $c_u$
  (cf. Lemma \ref{lem-ball-cu}) that, if $\Ome$ is a ball and $u = \chi_\Ome$
  then $\EE_0^{(3)}(u)< \infty$ for $q<d+2$ and infinite for $q\geq d+2$.  If
  $\Ome$ consists of multiple stripes, then $\EE_0^{(3)}(u) < \infty$ for all
  $q >d$.
  \end{remark}

  \subsection{Equivalent formulations of the energy}
  
In this section, we give some further geometric
representations of the limit energy $\EE_0$.  These representations are in
particular helpful for the calculation of the energy for specific
configurations.  We note that the assumption $\NPL{K_0}{1}{\R^d} = \infty$ is used in the proof of Lemma \ref{lem-decay} below. Hence, the results
do not hold for the approximating energies $\EE_\eps$.  Throughout the section
$F, G: \R_+\rightarrow \R_+$ are defined by
\begin{align}
    F(r) \ := \ \sig_{d-1}\int_r^\infty K_0(t) t^{d-2}\ dt, \qquad
    G(r) \ := \ \left |\int_1^r F(t) \ dt \right |, \qquad
\end{align}
Note that for $K_0(r)=r^{-d}$ we have $F(r)=\sig_{d-1}r^{-1}$ and
$G(r)=\sig_{d-1}|\ln r|$.

\medskip

We start with an auxiliary lemma: 
\begin{lemma}[Asymptotics for $r \to 0$] \label{lem-decay} %
  Let $u\in BV(\Td;\{0,1\})$. Suppose that \eqref{eq-molli-1}--
  \eqref{eq-molli-3} hold. If $\EE_0(u)<\infty$ then there are sequences
  $\delta_j, \gam_j\to 0$ such that
  \begin{enumerate}
  \item $\lim_{j\rightarrow \infty} F(\delta_j)\left (c_u(\delta_j)-c_u(0)-c'_u(0)\delta_j\right)=0$,\label{crit-seq1} 
  \item $\lim_{j \to \infty} G(\gam_j) (c'_u(\gam_j) - c'_u(0)) \ = \ 0$. \label{crit-seq2}
  \end{enumerate}
\end{lemma} 
\begin{proof}
  \textit{\ref{crit-seq1}:} With the representation
    \eqref{eq-EE1-auto} of the energy, we have
  \begin{align}\label{eq-bE1}
   \sig_{d-1} \int_0^1 (c_u(r)-c_u(0)-c'_u(0)r) K_0(r)r^{d-2} dr  %
    <  \EE_0(u) +  C_{K_0} \NI{c_u} \ %
    <  \infty,
  \end{align}
  where we also used \eqref{eq-molli-1}--\eqref{eq-molli-3}.  Arguing by
  contradiction we may assume that there exist $\delta_0>0$ and $c>0$ such that
\begin{align}\label{eq-F}
F(r)\left(c_u(r)-c_u(0)-c'_u(0)r\right)\geq c, \quad \text{ for all } r\in (0,\delta_0).
\end{align}
Since $F>0$ and $c_u(r)-c_u(0)-c'_u(0)r\geq 0$, from \eqref{eq-F} and
\eqref{eq-bE1} as well as the definition of $F$ we get
\begin{align} \label{eq-bEE2} %
  c\int_0^{\delta_0}\frac{-F'(r)}{F(r)} \ dr \ %
  \lupref{eq-F}< \ \EE_0(u) +  C_{K_0} \NI{c_u} \ %
  < \ \infty.
\end{align}
Since $F(\delta_0)<\infty$ by the assumption \eqref{eq-molli-3} of the kernel $K_0$, from \eqref{eq-bEE2} one has that
$\lim_{r\rightarrow 0}\ln F(r)<\infty$. This is a contradiction, because
$\lim_{r\rightarrow 0}F(r)\geq \int_0^1 K_0(r) r^{d-1}\ dr=\infty$.

\medskip

\textit{\ref{crit-seq2}:} Applying integration by parts to \eqref{eq-bE1} from $\delta_j$ to $1$ and using \ref{crit-seq1} we have
\begin{align}
-\lim_{j\rightarrow \infty}\int_{\delta_j}^1 \left(c'_u(r)-c'_u(0)\right) F(r)\ dr<\infty.
\end{align}
Analogously as for \ref{crit-seq1} and using that $\lim_{r\rightarrow 0} G(r)\geq \int_0^1 K_0(r) r^{d-1} \ dr =\infty$, one then gets a sequence
$\t\delta_j\rightarrow 0$ which satisfies \ref{crit-seq2}.
  \end{proof}
  
  With Lemma \ref{lem-decay} at hand we can express the energy using geometric
  quantities of $\Ome := \spt u$: \hkcc{We note that the assumption below in
    particular includes the case $K = r^{-d}$ except for $d = 1$. For $d = 1$,
    the case $K = \ln r$ is included (?)}
  \begin{proposition}[Equivalent formulations for $\EE_0$]
    \label{prp-equiv-energy} %
    Let $u\in BV(\Td;\{0,1\})$ with
    $\EE_0(u)<\infty$. Suppose that \eqref{eq-molli-1}--\eqref{eq-molli-3} hold.
    Let $\Ome := \spt u$ and let $H^-(x):=\{y\in \Rd: (y-x)\cdot \nu(x)\leq 0\}$ for
    $x\in \p\Ome$, where $\nu$ is the measure theoretic unit outer normal
    of $\Ome$ at $x$. We use the notation $z:=y-x$. Then we have
  \begin{enumerate}
  \item\label{itr-E1} %
    $\displaystyle \begin{array}[t]{ll}
       |\Td| \EE_0(u) \ &= \ - \ome_{d-1} P(\Ome)F(1)       \label{eq-energy_equi} \vspace{0.5ex}\\
                         &\displaystyle  \qquad+  \int_{\p \Ome} \int_{(\Ome\Delta H^-(x))\cap B_1(x)}\frac {F(|z|)}{|z|^{d-1}} \Big|\frac{z}{|z|}\cdot \nu(x) \Big| \  dy dx\\
                         &\displaystyle  \qquad+ \int_{\p \Ome}\int_{\Ome \cap B_1(x)^c  } \frac {F(|z|)}{|z|^{d-1}}  \left(\frac{z}{|z|} \cdot \nu(x)\right) \ dy dx.
    \end{array}$
  \item\label{itr-E2} %
    Let $\t\delta_j\to 0$ be any sequence as in Lemma
    \ref{lem-decay}\ref{crit-seq2}. Assume that there is a sequence
    $R_j\rightarrow \infty$ such that $G(R_j)R_j^{-\frac{d-1}{2}}\rightarrow 0$
    as $j\rightarrow\infty$. Then
    \begin{align} %
      |\Td| \EE_0(u) \ %
      &=  \ - \ome_{d-1}  P(\Ome)F(1)  \label{ene-formula} \\
      &\quad + \ \lim_{j \to \infty} \iint_{P_j} \frac{G(|z|)}{|z|^{d-1}}\left( \nu(y)\cdot \frac{z}{|z|}\right)\left( \nu(x)\cdot \frac{z}{|z|} \right)\ dydx, \notag
    \end{align}
    where
    $P_j := \ \{ (x,y) \in \p \Ome \times \p \Ome^{per} : \t\delta_j \leq
    |x-y| \leq R_j \}$ with $\Ome^{per} \SUS \R^d$ being the periodic copies
    of $\Ome$.
\end{enumerate}
\end{proposition}

\begin{proof}
  \textit{\ref{itr-E1}:} %
  We start with the representation \eqref{eq-EE1-auto} of $\EE_0$ in
    Proposition \ref{prp-energy-auto}.  We note that $F(R)\rightarrow 0$ as
    $R\rightarrow \infty$ which follows from \eqref{eq-molli-3}.  Integrating by
    parts both integrals in \eqref{eq-EE1-auto} (for the first integral of
    \eqref{eq-EE1-auto} we integrate by parts from $\delta_j$ to $1$, where the
    sequence $\delta_j$ is from Lemma \ref{lem-decay}\ref{crit-seq1}, and then let $\delta_j\rightarrow 0$) yields
  \begin{align}
    \EE_0(u) \ &= \  c'_u(0) F(1) + \int_{0}^1\left(c'_u(r)-c'_u(0)\right)F(r)\ dr + \int_1^\infty c'_u(r) F(r) dr  \\
    &=: \ I_1 + I_2 + I_3.  \label{eq-equi-2}
  \end{align}
 Here we have used that the integrand for $I_2$ is nonnegative and thus the limit exists. We note that $I_1= - \frac{\ome_{d-1} P(\Ome)}{\sig_{d-1}|\Td|}F(1)$.  To
  calculate $I_2$ and $I_3$ we first express $c'_u(r)$ using the geometric quantities. Using
  $\nabla u (x)=-\nu(x) \HH^{d-1}\lfloor \p^\ast \Ome$ and with
  $\Ome_r(x):=\frac{\Ome-x}{r}$, for
  a.e. $r>0$ we have
  \begin{align}
     c'_u(r) \ %
    &= \ \dashint_{\S^{d-1}} \dashint_{\Td} \ome\cdot \nabla u(y) u(y-r\ome)\ dy d\ome\\
    &= \ \frac{1}{\sig_{d-1}|\Td|}\int_{\p \Ome} \int_{\S^{d-1}\cap \Ome_r(x)} \ome \cdot \nu(x)\ d\ome dx.  \label{rep-c'} 
  \end{align}
  In the limit as $r\to 0$ we get
  \begin{align}
    c'_u(0) \ %
    &=  \ \frac{1}{|\Td|}\int_{\p \Ome}\int_{\S^{d-1}\cap H^-(x)}\ome\cdot \nu(x) \ d\ome dx. \label{rep-c'(0)}
  \end{align}
  Inserting \eqref{rep-c'} and \eqref{rep-c'(0)} into the integrand of $I_2$ and
  by Fubini we get
  \begin{align}
    I_2 = \frac{1}{|\Td|}\int_{\p \Ome} \int_0^1F(r)\left(\int_{\S^{d-1}\cap
    \Ome_r(x)}\ome\cdot \nu(x)d\ome- \int_{\S^{d-1}\cap
    H^-(x)}\ome\cdot \nu(x) d\ome\right) dx dr.    
  \end{align}
  Observing that for each $\nu\in \S^{d-1}$ fixed,
  $f(K):=-\int_{\S^{d-1}\cap K}\ome\cdot \nu\ d\ome $ achieves its maximum when
  $K=H^-(x)$, we have
  \begin{align}\label{eq-bend2} %
    I_2 \  = \  \frac{1}{|\Td|}\int_{\p \Ome} \int_0^1F(r)\int_{\S^{d-1}\cap
      \left(\Ome_r(x)\Delta H^-(x)\right)}|\ome\cdot \nu(x)|\ d\ome dr
      dx.
  \end{align}
  Similarly, inserting \eqref{rep-c'} into the integrand of $I_3$ we have
  \begin{align} \label{eq-bend3} %
    I_3 \  = \ \frac{1}{|\Td|}\int_{\p \Ome}\int_1^\infty F(r)\int_{\S^{d-1}\cap \Ome_r(x)} \ome \cdot \nu(x)\ d\ome dr dx.
  \end{align}
  The desired identity then follows from \eqref{eq-equi-2}, \eqref{eq-bend2} and
  \eqref{eq-bend2} with $\ome=\frac{y-x}{|y-x|}\in \S^{d-1}$ with $y\in B_1(x)$
  and $r=|y-x|$ and by a change from polar coordinates to Cartesian coordinates.

  \medskip
  
  \textit{\ref{itr-E2}:} We choose a sequence $\t \delta_j \to 0$ such that the
  convergence in Lemma \ref{lem-decay}\ref{crit-seq2} holds. Starting point is
  the representation \eqref{eq-equi-2} of the energy.  Integrating by parts and
  using that $G(1) = 0$ we have
\begin{align}
  \hspace{6ex} & \hspace{-6ex} %
                 \int_{\t \delta_j}^1 \big(c'_u(r)-c'_u(0)\big)F(r)  \ dr + \int_1^\infty c'_u(r) F(r) \ dr \\ %
                &= \  -\int_{\t \delta_j}^{R_j} c''_u(r) G(r)\ dr  - \left(c'_u(\t \delta_j)-c'_u(0)\right) G(\t \delta_j) - c_u'(R_j) G(R_j).
\end{align}
By our choice of sequence, the second term on the right hand side vanishes as
$j \to \infty$. By assumption as well as Proposition \ref{prp-auto2}\ref{itc-asym}\ref{itc-perim}  the term related to the evaluation at $R_j$
vanishes as $j \to \infty$. Adding the expressions for short-range and
long-range interaction together we
obtain the claimed equivalent form of the energy.
\end{proof}
\begin{remark}[Principal value interpretation of energy formula] \text{} %
  \begin{enumerate}
  \item When $K_0(r)=r^{-d}$, Proposition
    \ref{prp-equiv-energy}\ref{itr-E1} holds without the assumption
    $\EE_0(u)<\infty$. This is because 
    $\lim_{r\rightarrow 0+}\frac{c_u(r)-c_u(0)}{r}=c'_u(0)$ for
    any $u\in BV(\Td;\{0,1\})$ by Proposition \ref{prp-auto2}\ref{itc-perim} and $F(r)=\sig_{d-1}r^{-1}$, then Lemma \ref{lem-decay}\ref{crit-seq1} holds for any $\delta\rightarrow 0$ without the assumption $\EE_0(u)<\infty$.  Thus
    in this special case the energy formula \ref{itr-E1} holds for all
    $u\in BV(\Td;\{0,1\})$.
  \item The integrand in the energy formula \eqref{ene-formula} is not
    absolutely integrable in general if either $\t\delta = 0$ or
      $R = \infty$. An example for this is a configuration which consists of
    stripe domains, cf. Lemma \ref{lem-stripes}.
  \end{enumerate}
\end{remark}
\

\section{Proof of the theorems}\label{sec-proofs}

\subsection{Proof of Theorem \ref{thm-gamma}}\label{sec-Gamma}

We establish the compactness and $\Gam$-convergence of the nonlocal energy
$\EE_\eps(u)$ as $\eps \to 0$. We first establish the compactness result:
\begin{proposition}[BV-bound and compactness] \label{prp-gamma} %
  Let $K_\eps, K_0$ satisfy \eqref{eq-molli-1}--\eqref{eq-molli-3}.
      \begin{enumerate}
      \item For $u \in BV(\Td;\{0,1\})$ there are constants $C = C(K_0,d)>0$
        such that
  \begin{align}
    \NN{\nabla u} \ \leq \ C \left( |\Td|\EE_\eps(u)+ C\NP{u}{1}\right).
  \end{align}
\item For any family $u_\eps\in BV(\Td;\{0,1\})$ with
  $\sup_\eps \EE_\eps(u_\eps)<\infty$, there exists $u\in BV(\Td;\{0,1\})$
  and a subsequence (not relabelled) such that
    \begin{align}
      u_\eps \to u \text{ in } L^1 \quad %
      \text{ and } \quad %
      \NN{\nabla u_\eps}\to \NN{\nabla u} \qquad \text{for $\eps \to 0$}.
\end{align}
  \end{enumerate}
\end{proposition}
\begin{proof}
  \textit{(i):} We use the expression of the energy $\EE_\eps(u)$ in terms
  of the autocorrelation function in Proposition \ref{prp-energy-auto}.  By the
  assumptions \eqref{eq-molli-1}--\eqref{eq-molli-3} as well as $K_\eps\nearrow K_0$, there are $C_0, C_1 > 0$
  and a measurable set $A\subset (0,1)$, which depends on $K_0$, such that for
  sufficiently small $\eps >0$ we have
  \begin{align} \label{eq-Aeps} C_0 \leq \int_{A} K_\eps(r) r^{d-1}\ dr \qquad
    \text{and} \qquad %
    \int_{A\cup (1,\infty)} K_\eps(r)r^{d-2} dr\leq C_1. %
  \end{align}
  Using \eqref{eq-Aeps} as well as $c_u(r)-c_u(0)-c'_u(0)r \geq 0$ we then get
  \begin{align}
    -2c'_{u}(0) \int_{A}K_\eps(r)r^{d-1} \ dr \ %
    &\leq \ \sig_{d-1}^{-1}\EE_\eps(u) +  \int_{A\cup (1,\infty)} \left(c_{u}(0)-c_{u}(r) \right) K_\eps(r) r^{d-2}\ dr\\
    &\leq \ \sig_{d-1}^{-1}\EE_\eps(u)  + C_1 \NI{c_{u}}\leq \EE_\eps(u)  +C_1\frac{ \NPL{u}{1}{\Td}}{|\Td|},
  \end{align}
  where we have used Proposition \ref{prp-auto2}(ii) and \eqref{eq-Aeps} for the last
  estimate. Recalling the expression for $c'_u(0)$ in Proposition \ref{prp-auto2}(iv), we thus obtain the desired estimate from the above inequality.

  \medskip

  \textit{(ii):} By \textit{(i)} we have $\NN{\nabla u_\eps} \leq C$. By the compactness
  of the BV functions, there is a subsequence of $u_\eps$ (which we still denote
  by $u_\eps$) and $u\in BV(\Td,\{0,1\})$, such that $u_\eps\to u$ in $L^1(\Td)$
  and $\NN{\nabla u}\leq \liminf_{\eps\to 0} \NN{\nabla u_\eps}$. To show the
  convergence of $\NN{\nabla u_\eps}$, by Proposition \ref{prp-auto2} it
  suffices to show $c'_{u_\eps}(0)\to c'_u(0)$. Assume
  $c'_{u_{\eps_j}}(0)\to \alpha$ for some $\alpha\in \R$ along a subsequence. By
  the lower semi-continuity of the BV-norm we have $\alp \leq c'_u(0)$. It
  remains to show that $\alp \geq c'_u(0)$: By Fatou's lemma and
  \eqref{eq-Aeps} for any $\delta\in (0,1)$,
  \begin{align}
    \int_{\delta}^1\left(c_u(r)-c_u(0)-\alpha r\right) K_0(r)r^{d-2}\ dr \leq \liminf_{j\to  \infty}\sig_{d-1}^{-1}\mathcal{E}_{K_{\eps_j}}(u_{\eps_j}) +C_1\leq C.
  \end{align}
  On the other hand, by Proposition \ref{prp-auto2}
  \begin{align}
    \int_\delta^{1}\left(c_u(r)-c_u(0)-c'_u(0) r\right) K_0(r)r^{d-2}\ \ dr \ \geq \ 0.
  \end{align}
  Taking the difference of the above two inequalities yields
  \begin{align}
    \int_\delta^{1} \left(c'_u(0)-\alpha\right) K_0(r) r^{d-1} \ dr \ \leq \ C.
  \end{align}
  Since
  $\int_0^\infty K_0(r) r^{d-1} \ dr=\infty$ and $K_0\geq 0$, then $\int_\delta^\infty K_0(r) r^{d-1}\ dr \to \infty$
  as $\delta\to 0$. This implies that $c'_u(0)\leq \alpha$, particularly one has
  $-c'_u(0)\geq \limsup _{\eps\to 0} -c'_{u_\eps}(0)$. 
  \end{proof}

\medskip

As a consequence of Proposition \ref{prp-gamma}, we get the existence of
minimizers for $\EE_\eps$ with prescribed volume fraction:
\begin{proof}[Proof for Proposition \ref{prp-minex}]
Let $\eps\geq 0$ be fixed. It follows from Lemma \ref{lem-ball-cu} (ii)  that $\EE_\eps(u)<+\infty$ if $u$ is a stripe configuration, and furthermore $\EE_\eps(u)\geq -C$ for all $u\in BV(\Td;\{0,1\})$ (cf. Proposition \ref{prp-energy-auto}). Hence the infimum of $\EE_\eps(u)$ exists and is finite. Let $\{u_k\}$ be a minimizing sequence for $\EE_\eps$.  Then by Proposition \ref{prp-gamma}(ii) up to a subsequence $u_k\rightarrow u$ in $L^1(\Td)$ and $\|\nabla u_k\|\rightarrow \|\nabla u\|$ for some $u\in BV(\Td;\{0,1\})$. In particular $u$ satisfies \eqref{mass-constraint}.  Hence  $c_{u_k}\rightarrow c_u$ pointwisely and $c'_{u_k}(0)\rightarrow c'_{u_k}(0)$ by Proposition \ref{prp-auto2}\ref{itc-ucon} and \ref{itc-perim}.  Applying Fatou's lemma (in $(0,1)$) and dominated convergence theorem (in $(1,\infty)$) to the autocorrelation function formulation of $\EE_\eps$ in Proposition \ref{prp-energy-auto}, we obtain that $\EE_\eps(u)\leq \liminf_{k\rightarrow \infty} \EE_\eps(u_k)$. This implies that  $u$ is a minimizer. 
\end{proof}

At the end of the section we establish the $\Gam$-convergence of $\EE_\eps$:
\begin{theorem}[$\Gam$-limit] \label{thm-gamma_limit} %
  Suppose that \eqref{eq-molli-1}--\eqref{eq-molli-3} hold. Then $\EE_\eps \Gamto \EE_0$ in the
  $L^1$--topology, where $\EE_0$ is given in \eqref{E-0}. %
\end{theorem}
\begin{proof}
  We use the representation of $\EE_\eps$ from Proposition
    \ref{prp-energy-auto}, i.e.
  \begin{align}
    \EE_\eps(u) \ %
    = \ \sig_{d-1}\int_0^{\infty}\left(c_u(r)-c_u(0)-c'_u(0)r\chi_{(0,1)}(r)\right) K_\eps(r)r^{d-2}\ dr.
  \end{align}

  \textit{(i): Liminf inequality.}  We need to show that for any sequence 
  $u_\eps\to u$ in $L^1$ and $\sup_\eps\EE_\eps(u_\eps)\leq C<\infty$ we
  have
  \begin{align} \label{ls-toshow} %
    \EE_0(u) \ \leq  \ \liminf_{\eps\to 0} \EE_\eps(u_\eps).
  \end{align}
  By Proposition \ref{prp-gamma}(i) the sequence $u_\eps$ is uniformly bounded
  in $BV(\Td; \{0,1\})$. This together with the $L^1$ convergence of $u_\eps$
  and Proposition \ref{prp-gamma}(ii) yields that $u\in BV(\Td;\{0,1\})$ and
  $\|\nabla u_\eps\|\rightarrow \NNN{\nabla u}{}$. Then by Proposition \ref{prp-auto2} we have
  $c_{u_\eps}\to c_u$ uniformly and $c'_{u_\eps}(0) \to c'_{u}(0)$. Moreover,
  the integrands of the energy are nonnegative for $r\in (0,1)$. Thus by Fatou's
  lemma (applied to the integral over $(0,1)$) and the dominated convergence
  theorem (applied to the integral over $(1,\infty)$) one has
  \begin{align}
    \EE_0(u) \ %
      &\leq \  \sig_{d-1}\liminf_{\eps\to  0} \int_0^\infty
        \left(c_{u_\eps}(r)-c_{u_\eps}(0)-c'_{u_\eps}(0) r \chi_{(0,1)}(r) \right)
        K_\eps(r) r^{d-2}\ dr, \ %
  \end{align}
  which yields \eqref{ls-toshow}.

  \medskip
  
  \textit{(ii): Limsup inequality.}  Let $u\in BV(\Td,\{0,1\})$ with
  $\EE_0(u) < \infty$. Since $K_\eps\nearrow K$ by our assumption, by the monotone convergence theorem for the constant
  sequence $u_\eps := u$ we then obtain
  $\limsup_{\eps \to 0} \EE_\eps(u) \ \leq \EE_0(u)$.
\end{proof}

\subsection{Proof of Theorem \ref{thm-ball}}

In this section, we give the proof of Theorem \ref{thm-ball}, i.e., we show
that if the volume $\|u\|_{L^1(\Td)}=\lam |\Td|$ is sufficiently small depending on the
dimension, then the minimal energy is attained by a single ball. We give the
proof for a class of energies which in particular includes the energy
$\EE_0^{(1)}$ in Theorem \ref{thm-ball}. More precisely, we assume that there is
some $M=M(d)> 0$ such that for all $\eta\in (0,\frac{1}{2})$ we have
  \begin{align}\label{eq-kernel}
    \int_\eta^1 K_0(r) r^{d-1} dr \geq  M \max\left\{\frac{1}{\eta^2} \int_0^\eta K_0(r) r^{d+1} \ dr, \ \eta\int_\eta^1 K_0(r) r^{d-2}\ dr\right\}
  \end{align}
  We note that the assumption \eqref{eq-kernel} holds for
  $K_0^{(3)}(r)=r^{-q}$ for $d \leq q \leq d + \alp_0$ for
  some $\alp_0=\alp_0(d)>0$.  Moreover, \eqref{eq-kernel} together
  with \eqref{tay-ball} ensures that the energy of a single ball is finite.
  Then Theorem \ref{thm-ball} is a consequence of the next proposition:
  \begin{proposition}[Balls as minimizers for small volume] \label{prp-ball} %
    Suppose that \eqref{eq-kernel} holds. Then 
    $m_0=m_0(d,K_0)>0$ such that if 
    \begin{align}
      \lam \ \leq \ \min \Big \{ \frac{m_0}{|\Td|}, \frac{\ome_d}{4^d} \Big \},
    \end{align}
   then the unique, up to translation, minimizer $u$ of $\EE_0$ in $BV(\Td;\{0,1\})$
    with constraint \eqref{mass-constraint} is a single ball.
\end{proposition}
\begin{proof}
  The proof is based on a contradiction argument. Similar proofs have been made
  using the sharp isoperimetric deficit (cf. \cite{FMP}) e.g. in
  \cite{KnuMur1,KnuMur2,Jul,BonCri,FFMMM,MurZal}. We adapt these arguments
  to the autocorrelation function formulation of our model. The novelty in our
  proof is that the nonlocal term has the same scaling as the perimeter term.

  \medskip
  
  The \textit{isoperimetric deficit} of the bounded Borel set $\Ome\subset \Rd$
  is defined by
  \begin{align}
    D(\Ome) \ := \ \frac 1{|\p B|}|\p \Ome|-1, 
  \end{align}
  where $B \SUS \Rd$ is the ball with
  $|B|=|\Ome|$. The sharp estimate on the isoperimetric deficit in \cite{FMP}
  entails that
  \begin{align}\label{eq-hall}
    \hspace{-1ex}\min \left\{\frac{|\Ome\Delta B|}{|B|}: B\subset\Rd \text{ is a ball with} \ |B|=|\Ome|\right\} \ %
    \leq \ C_d \d^2,
  \end{align}
  where $\delta := \min \{1, D(\Ome) \}$, noting that the left hand side of
  \eqref{eq-hall} is called the \textit{Fraenkel asymmetry} of $\Ome$. 

  \medskip
  
  Let $u=\chi_{\Ome} \in BV(\Td;\{0,1\})$ be a minimizer in the class of
  functions $u\in BV(\Td;\{0,1\})$ which satisfy \eqref{mass-constraint}. Let
  $u_0=\chi_{B}$, where $B \SUS \Rn$ is the ball with $|B| = |\Ome|$ with
    radius
  $\rho := (\ome_d^{-1}\lam |\Td|)^{\frac 1d}\in
  (0,\frac{\ell}{4})$, which realizes the minimality in \eqref{eq-hall}.  By
  the definition of $\delta$ we have
  $\NN{\nabla u} \geq(1+\delta)\NN{\nabla u_0}$. If $\delta = 0$ there is
  nothing to prove, so that we assume $\delta > 0$ in the following noting that
  $\d = \d(\rho)$ in general.

  \medskip

  Let $\eta > 0$. By assumption \eqref{eq-molli-3} we have
  \begin{align}
    X_\eta \ := \int_\eta^1 K_0(r) r^{d-1} \ dr \ \to \ \infty \qquad \text{as $\eta \to 0$.}
  \end{align}
  Since $c_u(r)-c_u(0)-rc'_u(0)\geq 0$ and $K_0 \geq 0$, for $\eta \in (0, \rho )$
  we have
  \begin{align}
    \frac{\EE_0(u)}{\sig_{d-1}} \
    &\geq \ -  \int_{\eta}^1  c'_{u}(0) K_0(r)r^{d-1} \ dr -  \int_{\eta}^\infty (c_{u}(0)-c_{u}(r)) K_0(r)r^{d-2} \ dr\\
    &= \ -X_\eta c_u'(0)  -  \int_{\eta}^\infty (c_{u}(0)-c_{u}(r)) K_0(r)r^{d-2} dr \\
    &\geq \ -(1+\delta) X_\eta c_{u_0}'(0)  -  \int_{\eta}^\infty (c_{u}(0)-c_{u}(r)) K_0(r)r^{d-2} dr.
  \end{align}
  Using that $\EE_0(u)\leq \EE_0(u_0)$ this implies
  \begin{align}\label{eq-uu0} 
-\delta c'_{u_0}(0) X_\eta \ 
      &\leq \ \int_0^{\eta} (c_{u_0}(r)-c_{u_0}(0)-rc'_{u_0}(0))K_0(r)r^{d-2}\ dr \\
      &\qquad +  \int_{\eta}^{\infty} (c_{u_0}(r)-c_u(r))K_0(r)r^{d-2} \ dr \ %
    =: \ I_1 + I_2.
  \end{align}
  By Lemma \ref{lem-ball-cu} and since $\ell\geq 4\rho$, we have
  $|\Td| |c''_{u_0}(r)| \leq C r \rho^{d-3}$ for $r \leq \frac \rho 2$ and
  $|\Td| (\NI{c_{u_0}} + \rho \NI{c_{u_0}'}) \leq C \rho^d$ for
  $r \geq \frac \rho 2$. This implies
  \begin{align} \label{tay-ball} %
    |c_{u_0}(r)-c_{u_0}(0)-rc'_{u_0}(0)| \ %
    \leq \ \frac{C}{|\Td|} r^3 \rho^{d-3} \qquad\FA{r \geq 0}.
  \end{align}
  Using
  \eqref{tay-ball} we then get
  \begin{align} \label{tat-1} %
    I_1  \ %
    &\leq \  \frac{C}{|\Td|}\rho^{d-3} \int_0^{\eta} K_0(r)r^{d+1} dr.
  \end{align}
  By an application of Proposition \ref{prp-auto2}\ref{itc-ucon} we get
  \begin{align}
    I_2
    &\leq \ \frac {\NPL{u-u_0}{1}{\Td}}{|\Td|}  \int_\eta^{\infty} K_0(r) r^{d-2} \ dr \ %
      \leq \frac{C\delta^2\rho^d}{|\Td|} \int_\eta^\infty K_0(r) r^{d-2} \ dr,   \label{tat-2}
  \end{align}
  where we have used the estimate of the isoperimetric deficit \eqref{eq-hall}
  and the definition of $\delta$ in the last inequality. Combining
  \eqref{eq-uu0}, \eqref{tat-1} and \eqref{tat-2} and recalling $-c'_{u_0}(0)=\frac{\omega_{d-1}}{|\Td|}\rho^{d-1}$ we get
 \begin{align}
   \delta\rho^{d-1} X_\eta \ %
   \leq \ C \rho^{d-3} \int_0^\eta K_0(r) r^{d+1} \ dr + C \delta^2 \rho^d \int_\eta^\infty K_0(r) r^{d-2}\ dr.
 \end{align}
 Inserting the assumption \eqref{eq-kernel} to the above inequality and by \eqref{eq-molli-3}, we obtain
 \begin{align}
   \delta X_\eta \ %
   \leq \ \frac{C}{M}\left(\frac{\eta^2}{\rho^2} +\frac{\delta^2\rho}{\eta}\right)X_\eta+C_{d,K}\delta^2\rho.
 \end{align}
 With the choice $\eta := \delta \rho$, the above inequality gives
 $X_{\eta} \leq \frac{C(1+\delta)}{M}X_\eta + C_{d,K}\delta\rho$. Taking
 $M=4C$ and recalling that $\delta \leq 1$ we thus get $X_\eta \leq C_{d,K}\rho$.  Since $X_\eta\rightarrow \infty$ as
 $\eta\rightarrow 0$, then necessarily we have $\rho \geq \eta \geq C$ for some
 $C=C(d,K)>0$, and hence $\lam |\Td|=\omega_{d}\rho^d \geq c > 0$ for some
 constant $c =c(d,K)> 0$. We thus get a contradiction if
 $\lam|\Td| \leq m_0$ for some $m_0=m_0(d,K)$ sufficiently small.
\end{proof}
\begin{remark} Proposition \ref{prp-ball} is for small volume configurations
  instead of the volume fraction, as the smallness condition on $\lambda $ is
  not independent of $\ell$. Actually it follows from Lemma \ref{lem-small_disk}
  below that when the volume fraction $\lambda $ is sufficiently small, periodic
  balls have smaller energy than a single ball in $\Td$ for sufficiently large $\ell$.
\end{remark}

\section{Stripes and balls configurations} \label{sec-stripes-balls} %

In this section we explicitly compute the limit energy for stripes and lattice
balls for given volume fraction. Throughout this section we assume the kernel is
the prototypical $K_0^{(1)}(r)=r^{-d}$.  In the previous sections we have used
the technical assumption that the configurations are $\Td$--periodic for some
arbitrary periodicity $\ell$. To compare the energy for configurations with
fixed volume fraction, it is natural to consider more general configurations. We
hence write
\begin{definition}[Energy for generalized configurations]\label{defi-full}
  Let $u:\R^d\rightarrow \{0,1\}$ be a measurable function such that
    \begin{align}
      C_u(z) \ = \ \lim_{R\rightarrow \infty}\dashint_{[-R,R]^d}u(x+z)u(x) \ dx
    \end{align}
    exists. Then we set
    \begin{align} \label{def-gen} %
      \EE_0(u)=\int_{\R^d}\frac{C_u(z)-C_u(0)-z\cdot \nabla C_u(0)\chi_{B_1}(z)}{|z|} K_0(z)\ dz.
    \end{align}
\end{definition}
As noted in Remark \ref{rmk-general-pf}, this energy is independent of
the lattice $\Lambda$ where it arises from. In particular, the definition
  \eqref{def-gen} is consistent with and generalizes the definition of the energy
  $\EE_0$ in Theorem \ref{thm-gamma}(ii). We also note that we analogously get a more general definition of the
  energies $\EE_\eps$. However, these will not be used in this article.
\begin{remark}[Minimization in generalized configurations] \text{} %
 It is natural to study the minimization problem among generalized configurations with fixed volume fraction $\lambda = \lim_{\ell\rightarrow \infty} \dashint_{[-\ell,\ell]^d} u(x) \ dx$.  One could minimize $\EE_0$ in \ref{defi-full} in the space of autocorrelation functions
     \begin{align}
       \AA_\lam \ := \ \Big \{C_u:\R^d\rightarrow \R \ : \ C_u(0)=\lambda \}.
     \end{align}
Here $u$ is any function in $BV_{loc}(\R^d;\{0,1\})$ such that $C_u$ is well-defined. One can show that $\mathcal{A}_\lam$
     is convex.  Even though the limit energy is a linear functional in terms
   of autocorrelation functions, the problem still seems to be very hard since
   it is difficult to characterize the space $\mathcal{A}_\lam$.
 \end{remark}
In the rest of the section we calculate and compare the energy for stripe and ball
  configurations with the prototypical kernel $K_0^{(1)}(r)=r^{-d}$ (and the corresponding energy is denoted by $\EE^{(1)}_0$). We recall the generalized harmonic number $H_q$ with index $q \geq 0$ is given by
  \begin{align}\label{def-ghn} %
      H_q \ := \ \int_0^1 \frac{1-t^{q}}{1-t}\ dt.
  \end{align}
\begin{lemma}[Energy of stripes] \label{lem-stripes} \text{} %
  Let $d \geq 1$. Let $\Ome$ be a periodic set of stripes of
    width $d_0 >0$, distance $d_1 > 0$ and volume fraction
    $\lam := \frac{d_0}{d_0+d_1}$. More precisely, for $a := d_0+d_1$  let
     \begin{align} \label{Ome-stripe} %
       \Ome \ := \ \bigcup_{j\in \Z} ( j a, j a + d_0) \times \R^{d-1} .
     \end{align}
   Then the following holds:
       \begin{enumerate}
       \item The function $u := \chi_\Ome$ is $\T_a$--periodic and
   \begin{align}\label{eq-stripe}
        \EE_0^{(1)}(u) \ %
        = \  -\frac{2\lambda\omega_{d-1}}{d_0}\Big(1 + \frac{1}{2}H_{\frac{d-1}{2}}+ \ln \Big(\frac{d_0\sin(\pi\lambda)}{\pi\lambda}\Big)\Big),
   \end{align}
   where $H_{q}$ is the harmonic number (cf. \eqref{def-ghn}).
 \item Among configurations with prescribed $\lam \in (0,1)$ the minimal energy is
   \begin{align} \label{def-es} %
     e_S(\lam) \ := \ \min_{\text{$\lam$ fixed, stripes}} \EE_0^{(1)}(u) \ %
     = \ - 2 \ome_{d-1} e^{\frac{1}{2}H_{\frac{d-1}{2}}}
     \frac{\sin(\pi\lambda)}{\pi}.
   \end{align}
   with optimal width
   $d_{opt}= \frac{\pi\lambda}{\sin(\pi\lambda)}
   e^{-\frac{1}{2}H_{\frac{d-1}{2}}}$.
 \end{enumerate}    
  \end{lemma}
  \begin{proof}
  For the calculation we reduce the problem to the calculation of the
    interaction energy between two parallel slices.  Note that the configuration is $\T_a$--periodic. 

    \medskip

    \textit{(i):} %
    By Proposition \ref{prp-equiv-energy}\ref{itr-E2} the interaction energy
    between two parallel slices $\{0\}\times [0,a]^{d-1}$ and
    $ \{ \rho \} \times \R^{d-1}$, $\rho > 0$, with opposite outer normal reads
    \begin{align}  %
      I(\rho) \ %
      &:= \ \frac{1}{|\T_a|} \int_{\{0\}\times [0,a]^{d-1}}
      \int_{\{\rho\}\times\R^{d-1}} \frac{\ln \frac 1{|x-y|}}{|x-y|^{d-1}}
      \frac{y_1-x_1}{|y-x|} \frac{y_1-x_1}{|x-y|} \ dy dx \\ %
      &= \  - \frac{\rho^2}{2a } \int_{\R^{d-1}} \frac{\ln (\rho^2 +|y'|^2)}{(\rho^2 + |y'|^2)^{\frac{d+1}{2}}} \ dy', \ \label{def-Idef}
    \end{align}
    where we used that the inner integral in the first line above is independent
    of $x$ and $y_1-x_1=\rho$.  An explicit calculation yields \DET{ With
      the change of coordinates $|y'| = \rho r$ we then calculate
    \begin{align}
      I(\rho) \ &= \ - \frac 1a \ln \rho  \int_{\R^{d-1}} \frac{1}{(1 + |z|^2)^{\frac{d+1}{2}}} \ dz - \frac 12 \int_{\R^{d-1}} \frac{\ln(1+|z|^2)}{(1 + |z|^2)^{\frac{d+1}{2}}} \ dz \\%
                &=- \ \frac{\sig_{d-2}}{a}\ln\rho \int_0^\infty \frac{r^{d-2}}{(1+r^2)^{\frac{d+1}{2}}}\ dr - \frac{\sig_{d-2}}{2} \int_0^\infty \frac{\ln (1+r^2)}{(1+r^2)^{\frac{d+1}{2}}}r^{d-2} \ dr\\
    \end{align}
}
    \begin{align}
      I(\rho) %
      \ = \  - \frac{\sig_{d-2}}{2a } \int_0^\infty \frac{\ln \rho^2 + \ln (1 +t^2)}{(1 + t^2)^{\frac{d+1}{2}}} t^{d-2} \ dt \ 
      &= \ - \frac{\ome_{d-1}}{a} \left(\ln\rho+\frac{1}{2} H_{\frac{d-1}{2}}\right). \label{expr-I} %
    \end{align}
    By Proposition \ref{prp-equiv-energy}\ref{itr-E2} with $I(0) := 0$ and
    $I(-\rho):=I(\rho)$ for $\rho<0$ we get
    \begin{align}
      \EE_0^{(1)}(u) \ %
      &= \ -  \frac{\ome_{d-1}}{a^d}P(\Ome) + \frac{2}{a^d} \sum_{k \in \Z} \Big( I(ka+d_0) - I(ka) \Big) \\
      &= \ -\frac{\omega_{d-1}}{a^d}P(\Ome)
        + \frac{2}{a^d} \bigg(I(d_0) +\sum_{k=1}^\infty \big(I(ka+d_0)-2I(ka)+I(ka-d_0 \big)\bigg).
    \end{align}
    The infinite sum in the first line above is taken in the p.v. sense \DET{The
      following shows that the energy only depends on $a$, $\lam$: We need to
      show
      \begin{align}
        \ln d_0 + \sum_{k=1}^\infty \ln ( \frac{k^2-\lam^2}{k^2}) = \ln d_1 + \sum_{k=1}^\infty \ln (\frac{k^2-(1-\lam)^2}{k^2}).
      \end{align}
   And this follows from Euler's formula
    \begin{align}
    \frac{\sin(\pi\lam)}{\pi\lam}=\Pi_{n=1}^\infty \frac{k^2-\lam^2}{k^2},
    \end{align}
    as well as the relation $\frac{d_0}{d_1}=\frac{\lam}{1-\lam}$.}.  Inserting
  the formula for $I(\rho)$ and since $P(\Ome) = 2a^{d-1}$, we arrive at
    \begin{align}
      \EE_0^{(1)}(u) \ %
      &= \ -\frac{2\omega_{d-1}}{a}\bigg(\ln d_0+1 + \frac{1}{2}H_{\frac{d-1}{2}}  + \sum_{k=1}^\infty \ln \frac{(ka+d_0)(ka-d_0)}{(ka)^2}\bigg).
    \end{align}
    Using $d_0=\lambda a$ and the Euler's formula $\Pi_{n=1}^\infty \frac{k^2-\lam^2}{k^2}= \frac{\sin(\pi\lam)}{\pi\lam}$  we arrive at
  \begin{align}
    \EE_0^{(1)}(u) \ = \ %
    -\frac{2 \lambda\omega_{d-1}}{d_0}\bigg(\ln d_0+1+ \frac{1}{2}H_{\frac{d-1}{2}}+\ln \Big( \frac{\sin(\pi\lam)}{\pi\lam}\Big)\bigg).
  \end{align}

  \medskip

  \textit{(ii):} 
  The results  follows by a standard calculation. \DET{Consider
  \begin{align}
    f(a) := \frac{2\ome_{d-1}}{a}\left(c_1 +\ln (a \sin(\pi\lam))\right).
  \end{align}
  Then $f'(a) = 0$ leads to $ c_1 - 1 + \ln (a \sin(\pi\lam)) \ = \ 0$, i.e.
  \begin{align}
    a_{opt}   \ = \ \frac{e^{1 - c_1}}{\sin(\pi\lam)} \ %
    =: \ \frac{c_2}{\sin(\pi\lam)}.
  \end{align}
  Then
  \begin{align}
    \min_{\text{$\lam$ fixed, stripes}} \EE_0^{(1)}(u) \ = \ \frac{2\ome_{d-1} (c_1 +\ln c_2)}{c_2} \sin(\pi\lam).
  \end{align}
  This formular works for all vol fractions and we have $\EE_0^{(1)}(u) \to 0$ as
  $\lam \to 0$ or $\lam \to 1$. The prefactor should be negative for all
  $d$...}
\end{proof}
As a competitor we next consider a lattice of balls, arranged on a Bravais
lattice. We need to fix some notations: For a set of linearly independent
vectors $v_i \in \Rd$, $1 \leq i \leq d$ we consider the Bravais lattice
  \begin{align}
    \Lam \ := \ \Z v_1 \oplus \ldots \oplus \Z v_d  %
  \end{align}
 
  By $|\Lam|$ we denote the volume of the periodicity cell, i.e.
  \begin{align}
    |\Lam| \ := \ |\{ [0,1] v_1 \oplus \ldots \oplus [0,1] v_d \}|.
  \end{align}
  The energy for lattice balls can then be formulated in terms of the Appell series $F_4$, which
  in turn can be expressed using the Pochhammer symbols $(a)_n$, given by
  $(a)_n:=1$ for $n=0$ and $(a)_n:=a(a+1)\cdots (a+n-1)$ for $n\geq 1$. We set
\begin{align}\label{eq-Hd}
  \HH_d(t):= \ F_4 \Big(\tfrac{3}{2}, \tfrac{d+1}{2}; \tfrac{d+2}{2}, \tfrac{d+2}{2}; t, t\Big) %
  =  \sum_{m,n=0}^\infty \frac{(\frac{3}{2})_{m+n}(\frac{d+1}{2})_{m+n}}{(\frac{d+2}{2})_m (\frac{d+2}{2})_n m!n!}t^{m+n}.
\end{align}
With these notations, we have
  \begin{lemma}[Energy of ball configurations]\label{lem-small_disk} %
  Let $d\geq 1$. Let $\Lam \SUS \Rd$ be a Bravais lattice and let $\rho >
  0$. Let $\Ome=\bigcup_{q \in \Lam} B_\rho(q)$ be a periodic set of balls with
  centers on $\Lam$ with radius
    \begin{align} \label{r-small} %
      \rho \ \leq \ \frac 12 \min \big \{ \dist(p,q) \ : p,q \in \Lam \big \}
    \end{align}
   Then 
  \begin{enumerate}
  \item The function $u := \chi_\Ome$ is $\Lam$--periodic and we have
    \begin{align}
      \EE_0^{(1)}(u) \ = \ - \omega_{d-1} \Big(\ln (2\rho) + 1 -\frac{1}{2}H_{\frac{d-1}{2}}\Big)\frac{|\p B_\rho|}{|\Lam|}  + I_{\Lam}|B_\rho|^2 ,
    \end{align}
  where $H_q$ is defined in \eqref{def-ghn} and where 
  \begin{align} \label{lat-int} %
    I_{\Lam} \ := \ \frac{1}{|\Lam|} \sum_{q\in \Lam \BS \{0\}} \frac{\HH_d(\frac{\rho^2}{|q|^2})}{|q|^{d+1}},
    \end{align}
    where the Appell series $\HH_d$ is defined in \eqref{eq-Hd}.
  \item Let $\Lam_0$ be a fixed lattice with $|\Lam_0| = 1$ and let
      $e_{B,\Lam_0}(\lam)$ be the minimal energy among lattices of the
      form $\Lam = a \Lam_0$ for $ a> 0$ and for prescribed volume fraction
    $\lam = \frac{|B_\rho|}{|\Lam|}$. Then
    \begin{align} \label{def-eb} %
      e_{B,\Lam_0}(\lam)\ %
      = \ -(2 \omega_{d-1} de^{-\frac{1}{2}H_{\frac{d-1}{2}}}) \lambda +
      \sum_{e\in \Lam_0\setminus \{0\}}\frac{1}{|e|^{d+1}}
      \OO(\lambda^{2+\frac{1}{d}})
    \end{align}
    for all $\lambda \leq \lambda_0$, where $\lam_0 = \lam_0(\Lam_0)$ is the largest volume fraction
    which can be realized by balls in the lattice. The radius of the optimal ball
    configuration is given by
    $\rho_{opt}= \frac{1}{2}e^{\frac{1}{2}H_{\frac{d-1}{2}}} + \OO(1)$ as
    $\lam \to 0$. 
\end{enumerate}
\end{lemma}
\begin{proof}
  \textit{(i):} We introduce the full-space (radially-symmetrized)
  autocorrelation function for the single ball $\tilde{u}:=\chi_{B_\rho}$ by
  \begin{align}
    \t c_u(r) \ %
    := \ \dashint_{\S^{d-1}} \int_{\R^d} \t u(x+r\omega) \t u(x) \ dx d\omega
  \end{align}
  Using Lemma
    \ref{lem-interaction}, we decompose the energy as
  \begin{align}\label{eq:decomp}
    \EE_0^{(1)}(u) \ = \  \frac 1{|\Lam|}\Big(E_{\rm self}(\rho) + I_{\rm int}(\Lam,\rho)\Big),
  \end{align}
  Here,  $E_{\rm self}$ is the self--interaction energy of a single ball, i.e.
  \begin{align}
    E_{\rm self}(\rho)\ %
    := \ \sig_{d-1} \int_{0}^1\frac{\tilde c_u(r)-\t c_u(0)-\t c'_u(0)r}{r^2} \ dr + \sig_{d-1} \int_1^\infty \frac{\t c_u(r)-\t c_u(0)}{r^2}dr.
  \end{align}
  Furthermore, the interaction energy of a single ball $B_\rho$ with other
  copies is
   \begin{align}\label{eq-inter-0}
     I_{\rm int}(\Lam,\rho)   \ :=  \ \sum_{q\in \Lam \BS \{0\}}\int_{B_\rho(q)}\int_{B_\rho(0)}\frac{1}{|x-y|^{d+1} }\ dxdy
   \end{align}

  \medskip

  \textit{Computation of $E_{\rm self}(\rho)$:} By definition we have
  $\t c_{u}(r)=0$ for $r\geq 2\rho$.  Using Remark \ref{rmk-truncation} and
  integration by parts we have
  \begin{align}
    E_{\rm self}(\rho) \ %
    &= \ \sig_{d-1} \Big( \t c'_{u}(0)\ln(2\rho) + \int_0^{2\rho} \frac{\t c_{u}(r)-\t c_{u}(0)-r\t c'_{u}(0)}{r^2}dr - \frac{\t c_{u}(0)}{2\rho} \Big) \\
    &= \  \sig_{d-1}\t c'_{u}(0) \Big( \ln(2\rho) + 1 + \int_0^{2\rho}\frac{\t c'_u(r)-\t c'_u(0)}{\t c'_u(0) r}dr \Big). \label{dieda} %
  \end{align}
 Since $\t c_u(r) = |\Lam| c_u(r)$ for $0 \leq r\leq 2\rho$, in view
    of the formula in Lemma \ref{lem-ball-cu}(i) and with the change of
  variables $r = \rho \sqrt{1-s}$ we get
  \begin{align} \label{I1} %
    \int_0^{2\rho}\frac{\t c'_u(r)-\t c'_u(0)}{\t c'_u(0) r}dr \ %
    = \ - \frac 12 \int_0^1 \frac{1-t^{\frac{d-1}{2}}}{1-t}\ dt \ %
    \lupref{def-ghn}= \ - \frac 12 H_{\frac{d-1}{2}} %
  \end{align}
 Using the formula $\t c'_u(0) = -\frac{\ome_{d-1}}{\sig_{d-1}}|\p B_\rho|$ we arrive at
  \begin{align} \label{Eself-1} %
    E_{\rm self}(\rho) \ %
    &= \ \ome_{d-1} \left(\ln \frac 1{2\rho}-1+\frac{1}{2}H_{\frac{d-1}{2}}\right) |\p B_\rho|. 
  \end{align}

  \medskip
  
  \textit{Estimate of $I_{\rm int}(\Lam,\rho)$}:
  For the estimate of $I_{\rm int}(\Lam,\rho)$ we note that by \eqref{eq:lb_interaction} in Appendix \ref{sec:inter_balls} we have
    \begin{align}\label{eq-inter}
      \frac{ I_{\rm int}(\Lam,\rho)}{|\Lam|}  \ %
      =  \ |B_\rho|^2 I_\Lam.
    \end{align}
    Together with \eqref{Eself-1} and
    \eqref{eq:decomp}, this yields (i).
   
    \medskip

 \textit{(ii):} Using $\lam = \frac{|B_\rho|}{|\Lam|}$, we can further express $E_{\rm self}$ in terms of $\lam$ and $\rho$ as
 \begin{align} \label{Eself-2} %
   \frac{E_{\rm self}(\rho)}{|\Lam|} \ \lupref{Eself-1}= \ \frac{d\omega_{d-1}\lambda}{\rho} \Big(\ln \frac 1{2\rho} - 1 + \frac 12 H_{\frac{d-1}{2}}\Big),
 \end{align}
 We note that $I_\Lam = \frac 1{a^{2d+1}} I_{\Lam_0}$, where
   \begin{align}
   I_{\Lam_0}:=\sum_{e\in \Lam_0\setminus \{0\}}\frac{\HH_d(\frac{\tilde{\rho}^2}{|e|^2})}{|e|^{d+1}},\quad \tilde{\rho}=\frac{\rho}{a}.
   \end{align}
    Using this and
 $\lambda=\frac{|B_\rho|}{a^d}$ and $\ome_d \rho^d = \lam a^d$, the
 averaged interaction energy can be rewritten in terms of $\rho$ and $\lambda$ as
 \begin{align} \label{Iint-2} %
   \frac{I_{\rm int}(\Lam, \rho)}{|\Lam|} \ %
   = \ \frac{\lambda^2}{a} I_{\Lam_0} \ %
   = \ \frac{\lambda^{2+\frac 1d}}{\ome_d^{\frac{1}{d}} \rho} I_{\Lam_0}.%
  \end{align}
  From Lemma \ref{lem-inter} and that
  $\frac{\tilde{\rho}}{|e|}\leq \frac{1}{2}$, which follows from
  \eqref{r-small}, we have that
  $\sum_{e\in \Lam_0\setminus \{0\}}\frac{1}{|e|^{d+1}}\leq I_{\Lam_0}\leq
  C\sum_{e\in \Lam_0\setminus \{0\}}\frac{1}{|e|^{d+1}}$.
  Thus the self--interaction energy in \eqref{Eself-2} is of leading order in $\lambda$
  for $\lambda\ll 1$ for fixed lattice $\Lam_0$. By minimization \eqref{Eself-2}
  in $\rho$ we obtain
    \begin{align}
      \min_{\rho} \frac{E_{\rm self}(\rho)}{|\Lam|}  %
      =  \frac{E_{\rm self}(\rho_{\rm opt}^*)}{|\Lam|}
      =  -(2 \omega_{d-1} de^{-\frac{1}{2}H_{\frac{d-1}{2}}}) \lambda
      \quad %
      \text{with $\rho_{\rm opt}^*:=\frac{1}{2}e^{\frac{1}{2}H_{\frac{d-1}{2}}}$.}
    \end{align}
    Since expression \eqref{Iint-2} is of lower order in $\lam$, (ii) follows by
    a standard argument.  \DET{For fixed $\lam$ we want to minimize the energy.
    For small volume fraction we can ignore the $I$ term (independently of $a$).
    We hence consider
  \begin{align}
    \t f(a) \ = \ \frac{d \ome_{d-1}}{a \ome_d^{\frac 1d}}\Big(\t c_1 -\ln (a \lam) \Big), \ %
  \end{align}
  $\t f'(a) = 0$ leads to $\t c_1 - 1 = \ln (a \lam)$, i.e.
  $\t a_{opt} \ = \ \frac{e^{\t c_1 - 1}}{\lam} \ %
  =: \ \frac{\t c_2}{\lam}$. This yields
  \begin{align}
    \min_{\text{$\lam$ fixed, balls}} \EE_0^{(1)}(u) \ = \  \frac{d \ome_{d-1}}{\t c_2 \ome_d^{\frac 1d}}  \big(\t c_1 -\t c_2 \big) \lam %
  \end{align}
}
\end{proof}

We compare the minimal energy for
  stripes and balls in \textit{(ii)} of Lemma \ref{lem-stripes} and Lemma
  \ref{lem-small_disk} with fixed volume fraction $\lambda$.
  
  \begin{proposition}[Stripes vs balls] %
    Let $K_0^{(1)}(r)= r^{-d}$ and let $\Lambda_0$ be a Bravais lattice with $|\Lam_0|=1$. Let $e_S$, $e_{B,\Lam_0}$ be defined in \eqref{def-eb}
    and \eqref{def-es}. Then there are $\lam_0=\lam_0(d,\Lam_0)>0$ and
    $\d_0>0$ independent of $\Lam_0$ such that
    \begin{enumerate}
    \item  if $0 < \lam < \lam_0$, then $e_{B,\Lam_0}(\lam) \ < e_S(\lam)$.%
      \item if $d=2$ and $|\lam - \frac 12| < \d_0$, then $e_S(\lam) \ < e_{B,\Lam_0}(\lam)$.%
       \end{enumerate}
   \end{proposition}
  \begin{proof}
    \textit{(i):} 
   Comparing the leading order (in $\lambda$) terms of
    $e_S(\lambda)$ and $e_{B,\Lam_0}(\lambda)$ (which is independent of $\Lam_0$) 
    and using that
     \begin{align}
       \lim_{\lam \to 0} \frac {e_{S}(\lam)}{2 \ome_{d-1}\lam} \  %
       = \ - e^{\frac{1}{2}H_{\frac{d-1}{2}}} \ %
       > \ - de^{-\frac{1}{2}H_{\frac{d-1}{2}}} \ %
       = \ \lim_{\lam \to 0} \frac {e_{B,\Lam_0}(\lam)}{2 \ome_{d-1}\lam}%
     \end{align}
     for all $d\geq 2$ we can conclude the assertion. 

     \medskip

     \textit{(ii):} For large volume fraction the interaction energy is no more
    of lower order. Moreover, noting that balls have smaller self-interaction
    energy than stripes, we have to estimate the lower bound of interaction
    energy for balls. When $d=2$, by Lemma \ref{lem-stripes}(ii) and since
    $H_{\frac{1}{2}}=2-2\ln 2$ we have
  \begin{align}
  e_S(\frac{1}{2})=-\frac{2e}{\pi}.
  \end{align}
  Now we estimate the lower bound for the interaction energy among Bravais
  lattices in $\R^2$.  Given any Bravais lattice $\Lam \SUS \R^2$ with
  $a^2=|\Lam|$ and $\Lam_0:=\frac{1}{a}\Lam$, from
  \eqref{eq:lb_interaction} in Lemma \ref{lem-inter} and using that $\lam a^2=\pi \rho^2$ we
  have
\begin{align}
I_{\rm int} (\Lam, \rho)\geq \frac{|B_\rho|^2}{a^{5}} \sum_{e\in \Lam_0\BS \{0\}} \frac{1}{|e|^{3}}=\frac{\lam^{\frac{5}{2}}}{\sqrt{\pi}\rho}\sum_{e\in \Lam_0\BS \{0\}} \frac{1}{|e|^{3}}.
 \end{align}
Thus combining with the self-interaction energy in Lemma \ref{lem-small_disk} we have that for $u=\bigcup_{q\in \Lam} B_\rho(q)$ with $0<\rho\leq \frac{1}{2} \min_{p, q\in \Lambda}|p-q|$ and $\lam=\frac{|B_\rho|}{|\Lam|}=\frac{1}{2}$, 
  \begin{align}\label{eq-lb-2d}
 \EE_0(u)\geq -\frac{2}{\rho}\ln (4\rho) + \frac{\zeta(\Lam_0)}{4\sqrt{2\pi}\rho},\qquad \zeta(\Lam_0):= \sum_{e\in \Lam_0 \BS \{0\}}\frac{1}{|e|^3}.
  \end{align}
 By Rankin \cite{Rankin53}, $\zeta(\Lam_0)$ attains the minimum at the triangle lattice $H_{nor}:=\sqrt{\frac{2}{\sqrt{3}}}(\Z(1,0)\bigoplus \Z(\frac{1}{2}, \frac{\sqrt{3}}{2}))$ among all $2d$ Bravais lattices with volume $1$, and furthermore from its explicit expression we have that $\zeta(H_{nor})>8$.
 Thus \eqref{eq-lb-2d} together with the standard minimization   gives
 \begin{align}
\min_{\Lam_0} e_{B,\Lam_0}(\frac{1}{2}) > \min_{\rho} \Big(-\frac{2}{\rho}\ln (4\rho) + \frac{8}{4\sqrt{2\pi}\rho}\Big)> e_S(\frac{1}{2}).
 \end{align}
 Since the energy functionals are continuous in $\lambda$ , then there is $\delta_0>0$ independent of $\Lam_0$ such that if $\lambda \in (\frac{1}{2}-\delta_0, \frac{1}{2}+\delta_0)$ 
 stripes have strict smaller energy than any $2d$ lattice balls with the volume fraction $\lam$. 
 \end{proof}
 
 \begin{remark}[Triangular lattice in 2d] %
   We conjecture that in $2d$ for sufficiently small volume fraction $\lambda$,  the triangle lattice $H_{nor}$ has the smallest energy among all lattice of balls. Indeed, we recall from Lemma \ref{lem-small_disk}(i) that the energy consists of the self-interaction energy and the interaction energy. In view of \eqref{Eself-2}, the
   self-interaction energy is independent of the lattice $\Lam_0$. For the interaction energy, cf. \eqref{Iint-2}, we expect that
   $I_{\Lam_0}$ achieves its minimum for the triangle lattice $H_{nor}$ if
   $\lambda$ is sufficiently small. To prove such a result one would (at
     least) need to extend the methods in \cite{BK18} to the
     case of non--integrable potentials and we leave it for the future work. 
   \end{remark}

   \medskip
   
   \appendix
   
   \section{Connection to the results by D\'avila}\label{sec:davila}
   We start by recalling the classical result by D\'avila \cite{Da02}. Let $K_\eps:\R^d\rightarrow \R_+$ be a family of nonnegative radially symmetric kernels, which satisfy 
   \begin{align}\label{eq:assumption_davila}
     \lim_{\eps\rightarrow 0} \frac{\|K_\eps\|_{L^1(\R^d\setminus B_\delta)}}{\|K_\eps\|_{L^1(\R^d)}} \ = \ 0 \qquad \text{ for all  } \delta>0.
   \end{align}
 Let $\Omega\subset \R^d$ be a bounded set with finite perimeter. Then D\'avila showed that for $u:=\chi_{\Omega}$ one has
   \begin{align} \label{eq:first_order}
\lim_{\eps\rightarrow 0} \frac{1}{\|K_\eps\|_{L^1(\R^d)}}\int_{\R^d}\int_{\R^d} \frac{|u(x+z)-u(x)|}{|z|} K_\eps(z) \ dxdz =  \frac{2\ome_{d-1}}{\sig_{d-1}}\|\nabla u\|.
   \end{align}
  
   Using the autocorrelation function, one can give a simpler proof for
   \eqref{eq:first_order} as we explain below: Let
   \begin{align}
        \t c_u(r) \ := \ \dashint_{\S^{d-1}}\int_{\R^d}u(x+z)u(x) dz
   \end{align}
   be the (full space) radially symmetrized autocorrelation
   function. One can easily get the analogous statements as in Proposition
     \ref{prp-auto2} for the full space symmetrized autocorrelation
     function. Then it follows from assumptions \eqref{eq:assumption_davila} on
   $K_\eps$ as well as the existence of $\t c'_u(0)$, cf. Proposition
   \ref{prp-auto2}, that
   \begin{align}\label{eq:auto_bounded}
     \lim_{\eps\rightarrow 0} \frac{\sig_{d-1}}{\|K_\eps\|_{L^1(\R^d)}}\int_{0}^\infty \left(\t c_u(0)-\t c_u(r)\right) K_\eps(r) r^{d-2}  \ dr \ %
     = \ -\t c'_u(0).
   \end{align}
   With this at hand, \eqref{eq:first_order} directly follows from writing the
   above equation in terms of $u$ (cf. Proposition
   \ref{prp-auto2}\ref{itc-diffquot} and \ref{itc-perim} in the non-periodic
   setting).
 
We note that the decay of $K_\eps$ at infinity is not needed when $\Ome$ is bounded. Furthermore, instead of $\|K_\eps\|_{L^1(\R^d)}$ it suffices to normalize the LHS of \eqref{eq:first_order} by $\|K_\eps\|_{L^1(B_{r_0})}$ for arbitrary $r_0>0$.  
 
 \medskip
 
 Now we comment on our assumptions on the kernels $K_\eps$ as well as the
 convergence in \eqref{first-order}.  Since in this paper we are interested in
 the second order asymptotic expansion for the perimeter, we assume that
 $K_\eps\nearrow K_0$ as $\eps\searrow 0$ for some measurable function
 $K_0:\R^d\rightarrow \R_+$ with $\|K_0\|_{L^1(\R^d)}=\infty$,
 i.e. \eqref{eq-molli-1}.  Moreover, in the periodic setting one needs to
 consider the interaction of $\Ome$ with its periodic copies, consequently the
 autocorrelation function $\t c_u(r)-\t c_u(0)$ does not decay as
 $\rightarrow \infty$ but remains uniformly bounded.  Thus in view of
 \eqref{eq:auto_bounded} we assume $\int_1^\infty K_0(r) r^{d-2} \ dr<\infty$
 (cf. \eqref{eq-molli-3}) such that the energy functional is well-defined. Under
 these assumptions \eqref{first-order} holds true, which is an analogue of
 \eqref{eq:first_order} in the periodic setting and whose proof is a simple
 modification of that for \eqref{eq:first_order}. We also mention that the
 monotone convergence of $K_\eps$ is used when we prove the $\Gamma$-convergence
 of $\EE_\eps$ (cf. Theorem \ref{thm-gamma_limit}) to avoid the concentration
 effect. It might be possible to weaken this assumption, but this was not
   our aim in this paper.
 
   \section{Interaction energy of two balls}\label{sec:inter_balls}
   
   In this section we provide the explicit expression for the full-space
   interaction energy between two.
\begin{lemma}[Interaction energy between two balls in full
  space]\label{lem-inter}
  Let $q \in \Rd$. Then for $0<\rho\leq \frac{|q|}{2}$ we have
  \begin{align}
    \dashint_{B_\rho(q)}\dashint_{B_\rho(0)}\frac{1}{|x-y|^{d+1}} \ dxdy  \ %
    = \ \frac{1}{|q|^{d+1}}\HH_d \Big(\frac{\rho^2}{|q|^2}\Big),
  \end{align}
  where $\HH_d(t)$ is given in \eqref{eq-Hd}. In particular, for $C=C_d>0$
  we have
\begin{align}\label{eq:lb_interaction}
  \frac{3(d+1)}{d+2} \frac{\rho^2}{|q|^{d+3}} \ %
  \leq  \  \frac{1}{|q|^{d+1}}\HH_d \Big(\frac{\rho^2}{|q|^2}\Big)  - \frac 1{|q|^{d+1}}  %
  \leq  \ \frac{C\rho^2}{|q|^{d+3}}.
\end{align}
\end{lemma}
\begin{proof}
  By scaling and rotation invariance, we may assume that $\rho =1$,
  $\lam:=|q| \geq 2$ and $q=|q| e_1$. We note that the Fourier transform for
  $\chi_{B_1(0)}$ is 
\begin{align}
  \widehat{\chi_{B_1(0)}}(\xi) \ %
  := \ \frac{1}{(2\pi)^{\frac d2}}\int_{B_1(0)} e^{-i\xi\cdot x} \ dx \ %
  = \  |\xi|^{-\frac{d}{2}} J_\frac{d}{2}(|\xi|), \label{mu-char}
\end{align}
where $J_\nu$ is the Bessel function of first kind.  Using that
$B_1(0) =: B_1$ and $B_1(q) = B_1 + q$ are disjoint, we have
\begin{align}
X &:= \int_{B_1(q)}\int_{B_1(0)}\frac{1}{|x-y|^{d+1}} \ dxdy \\
&=  -\frac{1}{2}\int_{\R^d} \frac 1{|h|^{d+1}} \int_{\R^d} (\chi_{B_1}(x+h)-\chi_{B_1}(x))(\chi_{B_1+q}(x+h)-\chi_{B_1+q}(x))  \ dx dh.
\end{align}
By Plancherel and with the change of variable $h\mapsto \frac{h}{|\xi|}$ we
have
\begin{align}\label{eq:Plancherel}
  X &\ =\ -\frac{1}{2}\int_{\R^d}\int_{\R^d} e^{i\xi\cdot q}|\widehat{\chi_B}(\xi)|^2 \frac{|e^{i\xi\cdot h}-1|^2}{|h|^{d+1}}\ d\xi dh  \\
  &\ =\ -\frac{1}{2}\int_{\R^d}\int_{\R^d} e^{i\xi\cdot q}|\widehat{\chi_B}(\xi)|^2 |\xi| \frac{|e^{ih_1}-1|^2}{|h|^{d+1}}\ d\xi dh \\
    &\ =\  -c_d \int_{\R^d}e^{i\xi\cdot q}|\widehat{\chi_B}(\xi)|^2|\xi| \ d\xi,%
\end{align}
where the second equation of \eqref{eq:Plancherel} follows from the and where 
  \begin{align}
      c_d \ := \ \frac 12 \int_{\R^d}\frac{|e^{ih_1}-1|^2}{|h|^{d+1}}\ dh  \ = \ \frac{\pi^{\frac{d+1}{2}}}{\Gamma(\frac{d+1}{2})}.
  \end{align}
  Plugging \eqref{mu-char} into \eqref{eq:Plancherel} and with
$s := |\xi|\lam$, $\xi_1 = |\xi| \cos \theta$ and $t :=\cos\theta$ we get
\begin{align}
  X \ %
  &= \ - c_d\int_{\R^d}e^{i \xi_1 \lam}|J_{\frac{d}{2}}(|\xi|)|^2 |\xi|^{-d+1} \ d\xi\  \\
  & = - c_d\frac{\sig_{d-2}}{\lam} \int_0^\infty | J_{\frac{d}{2}}(\frac{s}{\lam})|^2 \int_{-1}^1e^{its} (1-t^2)^{-\frac{d-3}{2}} \ dt ds\\
  &=-\frac{c_d (2\pi)^{\frac{d}{2}}}{\lam} \int_0^\infty J_{\frac{d-2}{2}}(s)|J_{\frac{d}{2}}(\frac{s}{\lam})|^2 s^{-\frac{d-2}{2}}\ ds \ %
    =\frac{\pi^d}{\Gamma(\frac{d}{2}+1)^2\lam^{d+1}}\HH_d(\frac 1{\lam^2}), \label{this-c} %
\end{align}
where the last equality is due to \cite[eq. (7.1)]{Bailey35}. Taking the average
of $X$ by dividing $\ome_d^2$ and since
$\ome_d=\frac{\pi^{\frac{d}{2}}}{\Gamma(1+\frac{d}{2})}$, we obtain
\eqref{eq-Hd}. From the definition we see that $\HH_d(t)$ is monotonically
increasing in $t$ with $\HH_d(t) \geq 1+\frac{3(d+1)}{d+2}t$. This yields the
lower bound in \eqref{eq:lb_interaction}. The upper bound follows since
$\HH_d(t)=1+\OO(t)$ as $t\rightarrow 0$.
\end{proof}

\textbf{Acknowledgements:} HK and WS are very grateful to Felix Otto who brought
up the idea of using the autocorrelation function. This project started in
discussions with him. H.K. is a member of the Heidelberg STRUCTURES
  Excellence Cluster, which is funded by the Deutsche Forschungsgemeinschaft
  (DFG, German Research Foundation) under Germany’s Excellence Strategy EXC
  2181/1 - 390900948.

\renewcommand{\em}{\it} 

\bibliographystyle{plain}
\bibliography{periodic}

\end{document}